\documentclass[10pt]{amsart}

\usepackage{mathtools, amssymb, amsthm, tikz-cd, comment}
\usepackage[hidelinks]{hyperref}

\newtheorem{introthm}{Theorem}
\newtheorem{introcor}[introthm]{Corollary}
\newtheorem{introprop}[introthm]{Proposition}

\newtheorem{thm}{Theorem}[section]
\newtheorem{lem}[thm]{Lemma}
\newtheorem{prop}[thm]{Proposition}
\newtheorem{cor}[thm]{Corollary}

\theoremstyle{definition}
\newtheorem{defn}[thm]{Definition}

\theoremstyle{remark}
\newtheorem{rem}[thm]{Remark}

\numberwithin{equation}{section}

\newcommand{\bC}{{\mathbb C}}

\newcommand{\bF}{{\mathbb F}}

\newcommand{\bN}{{\mathbb N}}

\newcommand{\bR}{{\mathbb R}}

\newcommand{\bZ}{{\mathbb Z}}

\renewcommand{\Re}{\operatorname{Re}}
\renewcommand{\Im}{\operatorname{Im}}

\newcommand{\ee}{\varepsilon}

\theoremstyle{plain}

\theoremstyle{definition}

\numberwithin{equation}{section}

\title{Some remarks on decay in countable groups and amalgamated free products}

\author{Srivatsav Kunnawalkam Elayavalli}
\address{Srivatsav Kunnawalkam Elayavalli, 
Department of Mathematics, University of California, San Diego, 9500 Gilman Drive \#0112, La Jolla, CA 92093, USA}
\email{skunnawalkamelayaval@ucsd.edu}

\author{Gregory Patchell}
\address{Gregory Patchell, 
Mathematical Institute, University of Oxford, Andrew Wiles Building, Radcliffe Observatory Quarter, Woodstock Road, Oxford, OX2 6GG, UK}
\email{greg.patchell@maths.ox.ac.uk}

\author{Lizzy Teryoshin}
\address{Lizzy Teryoshin, 
Department of Mathematics, University of California, San Diego, 9500 Gilman Drive \#0112, La Jolla, CA 92093, USA}
\email{eteryoshin@ucsd.edu}

\date{\today}

\begin{document}
	
\begin{abstract}

In this note, we first study the notion of \emph{subexponential decay} (SD) for countable groups with respect to a length function, which generalizes the well-known rapid decay (RD) property, first discovered by Haagerup in 1979. Several natural properties and examples are studied, especially including groups that have SD, but not RD. This consideration naturally has applications in $C^*$-algebras. We also consider in this setting a permanence theorem for decay in amalgamated free products (proved also recently by Chatterji--Gautero), and demonstrate that it is in a precise sense optimal. 
\end{abstract}

\maketitle

\section{Introduction}

Let $G$ be a countable group. One considers length functions on $G$, namely $l:G\to \mathbb{N}$ satisfying $l(e)=0$, $l(g_1)= f(g_1^{-1})$, $l(g_1g_2)\leq l(g_1)+l(g_2)$ for all $g_i\in G$. Other than the trivial one, there are natural length functions on finitely generated groups, arising from the word length. For a group $G$ equipped with a length function $l$, one considers notions of decay, which computes the distortion of the operator norm of functions in the group ring acting by the left regular representation, relative to the $2$-norm, i.e, $\|\phi\|_2= \left(\sum_{g\in G} |\phi(g)|^2\right)^{1/2}$. In particular, we say that $(G,l)$ admits $f$-decay for a function $f:\mathbb{N}\to \mathbb{R}$ if for all $\phi\in \mathbb{C}[G]$ supported on the ball of radius $n$ centered at $e$ with respect to $l$ (denoted $B_l(e)$), we have $\|\phi\|\leq f(n)\|\phi\|_2$. If $G$ admits a length function $l$ such that $(G,l)$ admit $f$-decay where $f$ is a polynomial, then it is said that $G$ has the \emph{rapid decay property} (hereafter referred to as RD). This property was originally proved for $\mathbb{F}_2$ by Haagerup in \cite{haagerup1978example}. Jolissaint systematized this in his work \cite{jolissaint1990rapidly}, and since then several examples of interest have been obtained. Notably, this includes all hyperbolic groups \cite{delaharpe}, mapping class groups \cite{behrstockcentroid}, cocompact lattices in some higher rank groups such as $SL_3(\mathbb{R})$ \cite{RRS1998higherrankRD,lafforgueRD,chatterji2001property, Chatterji}, right angled Artin groups \cite{graphprodRD}, etc. For a beautiful survey, please consult the work of Chatterji \cite{Ch17}.

With regard to analytic considerations on groups, the notion of rapid decay has had an extraordinary impact. For instance using RD, Haagerup proved that $C^*_r(\mathbb{F}_2)$ has the metric approximation property \cite{haagerup1978example}; Lafforgue settled the Baum--Connes conjecture for examples of groups with Kazhdan's property \cite{LafforgueBaumConnes}; Grigorchuk and Nagnibeda studied generalizations of growth functions in groups and properties including rationality \cite{Grigorchuk}; building on works of Robert \cite{robertselfless}, and Louder--Magee \cite{louder2022strongly}, Amrutam, Gao, and the first two authors proved strict comparison for $C^*_r(G)$ for a large family of countable groups \cite{amrutam2025strictcomparisonreducedgroup} (see also \cite{elayavalli2025negativeresolutioncalgebraictarski, vigdorovich2025structuralpropertiesreducedcalgebras}). 

Motivated in part by the above considerations, in  this paper we first investigate for countable groups $G$ a natural weakening of RD, namely admitting a length function $l$ such that $(G,l)$, admits $f$-decay where $f$ is a \emph{subexponential} function, i.e, $\liminf_{n\to \infty} f(n)^{1/n}= 1$. We call this \emph{property SD}. Very much like RD, this property satisfies several permanence properties, such as taking subgroups, direct products, free products and generally graph products (see Section \ref{sect: graph-prod-SD}). In the case of finitely generated amenable groups, SD is equivalent to having subexponential growth (see Proposition \ref{prop: amen-subexp-growth-is-SD}). Hence, a natural obstruction to SD is again the existence of finitely generated amenable subgroups with exponential growth. Clearly all groups with RD admit SD, but the converse is not true, as one considers the family of finitely generated amenable groups with intermediate growth in the sense of Grigorchuk \cite{grigorchuk2008groups}. However, it is natural to ask for countable groups with SD that admit no finitely generated amenable subgroup with super-polynomial growth, and in addition do not enjoy RD. We show that such groups exist, and are in fact minor modifications of a construction of Sapir \cite{sapir2015rapid} (see Theorem \ref{thm: sapir-sd-not-rd}). 

It turns out that in many cases, the more general SD seems to be sufficient to guarantee applications (where RD has been used). For instance, this is rather immediate in the case of the MAP property (see \cite{haagerup1978example}), and also in the works of Grigorchuk--Nagnibeda \cite{Grigorchuk}. In the case of the work of Lafforgue on the Baum--Connes conjecture, it is plausible that rapid decay may be replaced by subexponential decay. However, we are unable to settle this question completely at the moment (see Remark \ref{frechet-bad}). In \cite{amrutam2025strictcomparisonreducedgroup}, RD was crucially used for proving selflessness in the sense of Robert for $C^*_r(G)$. Unfortunately, it is not possible to simply replace RD with SD in the proof of Theorem 3.5 of \cite{amrutam2025strictcomparisonreducedgroup}, simply because a composition of a subexpoenential function with a subexponential function need not be subexponential (taking $f(x)=e^{\sqrt{x}}$ and $g(x)=x^2$, then $f\circ g$ is a counterexample). Fortunately, we found a rather subtle replacement of the rapid decay argument involving noncommutative Khintchine-type inequalities of Ricard--Xu \cite{KhinRX} instead (inspired by recent work \cite{hayes2025selflessreducedfreeproduct}). Using this we are able to bypass this obstruction and show that $C^*_r(G_1*G_2)$ is selfless whenever $G_1,G_2$ are infinite groups that have SD, and one of them contains an infinite order element. This in particular recovers selflessness for $C^*_r(G*G)$ where $G$ is the well known Grigorchuk's group \cite{grigorchuk1985degrees}. Note that N. Ozawa in \cite{ozawa2025proximalityselflessnessgroupcalgebras} has very recently proved that all groups admitting extremely proximal boundaries are $C^*$-selfless (this includes all non-degenerate free products).  Despite the fact that his result is much more general, we document our different approach here because it might be useful for other future directions (see Section \ref{sect: selfless}).

Recall the construction of amalgamated free products $G_1*_AG_2$ where $A$ is a common subgroup with fixed embeddings into $G_1$ and $G_2$. This class of groups is quite a robust generalization of free products. Preservation of RD was proved in certain amalgamated free products in the case of finite (or central finite index) subgroups in \cite{jolissaint1990rapidly}. The situation in the general setting of amalgamated free products, as experts are aware, is rather tricky. Indeed there are plenty of examples of amalgamated free products that do not preserve RD, for instance even in the case of finite index amalgams. It is known that $SL_2(\mathbb{Z}[\frac{1}{p}])\cong SL_2(\mathbb{Z})*_{A}SL_2(\mathbb{Z})$ where $A$ is a finite index subgroup of $SL_2(\mathbb{Z})$. However, $SL_2(\mathbb{Z})$ has RD while $SL_2(\mathbb{Z}[\frac{1}{p}])$ does not since it is a non uniform lattice in a higher rank group (see Remark 8.6 in \cite{ValetteBC}). Another subtle example, which will be of importance to us, is $\mathbb{Z}^2\rtimes SL_2(\mathbb{Z})$ which is a non degenerate amalgamated free product where the amalgam, namely $\mathbb{Z}/2\mathbb{Z}$, is virtually nilpotent unlike in the previous example.
\subsection*{An important remark:}
We now state some results concerning permanence of decay in amalgamated free products, and explain their optimality by considering some key examples. Upon sharing our draft with experts, I. Chatterji kindly made us aware of their very recent paper \cite{chatterji2024distortiongraphsgroupsrapid}, which obtains permanence results concerning rapid decay, in the general setting of fundamental groups of graphs of groups. Their results have a significant overlap with our results below, in particular, Proposition 1.3 of \cite{chatterji2024distortiongraphsgroupsrapid} and its elegant proof recovers our Theorem \ref{thm:afp-distortion} (1) below. However, our work was carried out independently and the proof is of a different nature. It is in a certain sense more local, and combinatorial and does not take Jolissaint's free product permanence as a black-box. Apart from working in the framework of SD, an aspect of our work that is new and not covered in \cite{chatterji2024distortiongraphsgroupsrapid} is that we highlight some concrete situations such as those in Theorem \ref{thm: no distortion} and Corollary \ref{thm: amalgam-distortion}, and also demonstrate the optimality of this result by highlighting certain examples of interest with exponential distortion.   For these reasons and the possibility that our proofs may have some intrinsic value for future considerations, we felt it worthwhile to still include our presentation in the literature. First we present the following result.

\begin{introthm}
\label{thm: no distortion}
        Suppose $(G,L_G)$ and $(H,L_H)$ have RD (resp. SD). Suppose $A$ is a common subgroup of $G$ and $H$ such that $L_G|_A = L_H|_A$. Then $\Gamma=G*_AH$ has RD (resp. SD). Moreover, if $G,H$ both have $f$-decay then $G*_AH$ has $\tilde{f}$-decay where $\tilde{f}(x) = (2x+1)^{7/2}f(2x)$.
\end{introthm}

The length function (denoted $L$ for simplicity) on $\Gamma$ that witnesses RD (resp. SD) in the above result is defined as follows (Definition \ref{def: afp-len-agree-on-A}): 
    $L(k) := \min\{L_A(a) + L_G(g_1) + L_H(h_1) + \ldots \mid
    k = ag_1h_1\cdots, \ a\in A, g_i\in G\setminus A, h_i\in H\setminus A\}$. This is indeed shown to be a length function in Lemma \ref{lem: L-well-def}. We would first like to remark that this yields various examples of groups with RD, including arbitrary group doubles of RD groups. Note also that the word length functions agree when restricting to the amalgam in a graph product, while decomposing it as an amalgamated free product in the natural sense. Therefore combining this result with an inductive argument one recovers permanence of RD under graph products (recovering with a new proof, the main result of \cite{ciobanu2011rapiddecaypreservedgraph}). 

A natural question here is how far can one push the above result in the event that $L_G$ does not coincide with $L_H$ on $A$. In order to address this first we need to set up some definitions. We define a \emph{universal} length function called $L^U$ on the amalgamated free product $G*_A H$ as follows:
    $$L^U(k) = \min\left\{\sum_i L_i(k_i) : \prod_i k_i = k, k_i\in G\cup H\right\},$$
    where $L_i = L_G$ if $k_i\in G\setminus A$, $L_i= L_H$ if $k_i\in H\setminus A$, and if $k_i$ is in $A$, then $L_i$ can be chosen either to be $L_G$ or $L_H$ (this is key). Now, since $L^U$ is defined on all of $G*_AH$, its restriction to $G$ and $H$ now gives length functions which agree on $A$.  Say two length functions $(L_1,L_2)$ on a group $K$ are $f$-distorted if $L_1(k)\le f(L_2(k)) $ and $L_2(k)\le f(L_1(k))$ for all $k\in K$. We prove the following (note that (1) below was first proved independently by \cite{chatterji2024distortiongraphsgroupsrapid}): 

    \begin{introthm}
    \label{thm:afp-distortion}
           Let $(G,L_G)$ and $(H,L_H)$ be groups with length functions and a common subgroup $A$. If $G,H$ both have $f$-decay and if $(L^U|_G,L_G)$ on $G$ and $(L^U|_H,L_H)$ on $H$ are $g$-distorted, then $G*_AH$ has $\tilde{f}\circ g$-decay, where $\tilde{f}(x) = (2x+1)^{7/2}f(2x)$. In particular, we have the following cases.
    \begin{enumerate}
        \item RD if $(G,L_G)$ and $(H,L_H)$ have RD and the pairs $(L^U|_G, L_G)$ on $G$ and $(L^U|_H,L_H)$ on $H$ are polynomially distorted;
        \item SD if $(G,L_G)$ and $(H,L_H)$ have RD and the pairs $(L^U|_G, L_G)$ on $G$ and $(L^U|_H,L_H)$ on $H$ are subexponentially distorted;
        \item SD if $(G,L_G)$ and $(H,L_H)$ have SD and the pairs $(L^U|_G, L_G)$ on $G$ and $(L^U|_H,L_H)$ on $H$ are linearly distorted.
    \end{enumerate}
    \end{introthm}

    First we make some preliminary remarks. In the case that $L_G$ and $L_H$ are the word lengths on $G$ and $H$, then $L^U$ is the word length on $G*_AH$. Furthermore, $L^U = L$ (as in Definition \ref{def: afp-len-agree-on-A}) if $L_G$ and $L_H$ agree on $A$. Indeed, $L^U\le L$ is immediate; conversely, if $k = \prod k_i$ then $L(k) \le \sum L(k_i) = \sum L_i(k_i)$ and so minimizing over all choices of $k_i$ gives $L^U(k)$. Note that by Lemma \ref{lem: L = L_G}, one has in this case $L$ agrees with $L_G$, therefore Theorem \ref{thm:afp-distortion} indeed recovers Theorem \ref{thm: no distortion}.
    We also point out that this recovers, in particular, the specific classes of amalgamated free products $G*_AH$ considered by Jolissaint \cite{jolissaint1990rapidly}. In the case $A$ is finite, one can boundedly distort the length function on $A$ to 0 while preserving RD/SD in $G$ and $H$. In the case $A$ is central and finite index, one can distort the length on $G$ and $H$ by a linear amount such that they agree on $A$ and extend to a length function on $G*_AH$ (see Page 181 in \cite{jolissaint1990rapidly}). Therefore, for both such examples of amalgams, if $G$ and $H $ have SD then $G*_AH$ also has SD.

    Now we explain in what sense the result above is optimal. The example given earlier in the introduction of $SL_2(\bZ[1/p])$ is illustrative of the necessity of subexponential distortion. We refer the reader to Section II.1 of \cite{serre2003trees} and pages 4--5 of \cite{dogon2025connections} for the following facts about $SL_2(\bZ[1/p])$. First, $SL_2(\bZ[1/p]) \simeq SL_2(\bZ)*_A SL_2(\bZ)$ where $A$ is finite index in both copies of $SL_2(\bZ).$ Moreover, it is known that $SL_2(\bZ[1/p])$ has exponential growth solvable and therefore amenable subgroups, and so does not have RD. However, $SL_2(\bZ)$ does have RD by virtue of containing $\bF_2$ as a finite index subgroup. 
    
    Accordingly, the word length on $SL_2(\bZ[1/p])$ is exponentially distorted relative to the word length on $SL_2(\bZ)$. Indeed, there are matrices $a\in A $ and $g_1,g_2$ in the two copies of $SL_2(\bZ)$ respectively such that $a^{p^2} = g_2g_1ag_1^{-1}g_2^{-1}$. Explicitly, we can take $a = \begin{pmatrix}
        1 & 1\\ 0& 1
    \end{pmatrix}$, $g_1 = \begin{pmatrix}
        0&-1\\1&0
    \end{pmatrix}$, and $g_2 = \begin{pmatrix}
        0 & -p\\ p^{-1} & 0
    \end{pmatrix}$. (Note that $g_2$ is actually the copy of $g_1$ in the second copy of $SL_2(\bZ)$ inside of $SL_2(\bZ[1/p])$.) This means that the length of $a^n$ must grow at most logarithmically in $SL_2(\bZ[1/p])$; however, it is well-known that in $SL_2(\bZ)$ the length of $a^n$ grows linearly. Thus RD (and SD) fail even just on the subgroup of $A$ generated by $a$ when considering word length in $SL_2(\bZ[1/p])$.

    For another explicit example inspired by Lemma 1 of \cite{lubotzky1993expgrowth}, consider $\bZ^2 \rtimes SL_2(\bZ)$ which has amalgamated free product decomposition $\bZ^2\rtimes \bZ/6\bZ * _{\bZ^2\rtimes \bZ/2\bZ} \bZ^2\rtimes \bZ/4\bZ$. The generators here can be made very explicit; let $U = \begin{pmatrix}
        0 & -1 \\ 1 & 1
    \end{pmatrix}$ be the order 6 generator and $S=\begin{pmatrix}
        0&-1\\1&0
    \end{pmatrix}$ be the order 4 generator. Note that when equipped with the word length generated by $\begin{pmatrix}
        1\\0
    \end{pmatrix}$ and $\begin{pmatrix}
        0\\1
    \end{pmatrix}$, $S$ does not at all distort $\bZ^2$ while $U$ does slightly: $U \cdot \begin{pmatrix}
        0\\ n
    \end{pmatrix} = \begin{pmatrix}
        -n \\ n
    \end{pmatrix}$, so $U$ can distort vectors by a factor of 2. This is sufficient to achieve exponential distortion in the amalgamated free product. Indeed, consider $A = SUSU^2 = \begin{pmatrix}
        2&1\\1&1
    \end{pmatrix}$ and $B = SU^2SU = \begin{pmatrix}
        1&1\\1&2
    \end{pmatrix}$. Then $A\cdot \begin{pmatrix}
        n\\n
    \end{pmatrix} + B\cdot \begin{pmatrix}
        n\\n
    \end{pmatrix} = \begin{pmatrix}
        3n \\3n
    \end{pmatrix}$. Thus the length of $\begin{pmatrix}
        3^n \\ 3^n
    \end{pmatrix}$ grows only linearly in $\bZ^2\rtimes SL_2(\bZ)$, contradicting RD and SD as expected.

It is quite natural to also wonder about the local distortion of the restrictions of $L_G$ and $L_H$ on $A$, which apriori seems more transparent to work with than $L^U$. The previous examples of $SL_2(\bZ[1/p])$ and $\bZ^2\rtimes SL_2(\bZ)$ show that even a small amount (e.g., linear) of distortion in the amalgam can yield rather drastic outcomes for $L^U$. However, if the distortion on the amalgam is very well-controlled then we can still make conclusions about the distortion in $G*_AH$.

\begin{introprop}
\label{amalgam-distort-bounds-afp-distort}
    Let $(G,L_G)$ and $(H,L_H)$ have common subgroup $A$. Suppose that $(L_G|A,L_H|A)$ are $f$-distorted on $A$. Then $(L^U|G,L_G)$ on $G$ and $(L^U|_H,L_H)$ are $\tilde{f}$-distorted where $\tilde{f}(x) = x\exp(f(x)-x)$. In particular, we have the following cases:
    \begin{enumerate}
        \item If $f(x) - x = O(1)$, then $\tilde{f}$ can be taken linear;
        \item If $f(x) -x = O(\ln(x))$, then $\tilde{f}$ can be taken polynomial;
        \item If $f(x) -x = o(x)$, then $\tilde{f}$ can be taken subexponential.
    \end{enumerate}
\end{introprop}

Combining the above result with Theorem \ref{thm:afp-distortion} we therefore immediately obtain the following. 

\begin{introcor}
\label{thm: amalgam-distortion}
    Let $(G,L_G)$ and $(H,L_H)$ have common subgroup $A$. Suppose that $(L_G|A,L_H|A)$ are $g$-distorted on $A$ and that $G,H$ have $f$-decay. If $g(x) = x + j(x)$ then $G*_AH$ has $\tilde{f}\circ (x\exp(j(x)))$-decay where $\tilde{f}(x) = (2x+1)^{7/2}f(2x)$. In particular, we have the following cases:
    \begin{enumerate}
        \item If $g(x) - x = O(\ln(x))$ and $G,H$ have RD then $G*_AH$ is RD;
        \item If $g(x) - x = o(x)$ and $G,H$ have RD then $G*_AH$ is SD;
        \item If $g(x) - x = O(1)$ and $G,H$ have SD then $G*_AH$ has SD.
    \end{enumerate}
\end{introcor}

We would now like to describe some aspects of the proofs. In the key case that $L_G$ and $L_H$ agree on their common subgroup $A$, the natural length $L$ defined on the amalgamated free product $G*_AH$ agrees with $L_G$ and $L_H$ on $G$ and $H$ respectively, see Lemmas \ref{lem: L = L_G} and \ref{lem: L-well-def}. In this case, we also have a nice formula for $L$ given by minimizing over elements in $A$, see Equation \ref{afp-len-iden}. Now, a careful analysis allows us to proceed in the footsteps of Jolissaint and Haagerup \cite{jolissaint1990rapidly,haagerup1978example}. It suffices to prove an estimate on group algebra elements $\varphi,\psi$ where $\varphi$ has controlled support with respect to $L$, and where we restrict the reduced word length (or number of syllables) in the supports of $\varphi,$ $\psi,$ and $\varphi*\psi$. We use the fact that in amalgamated free products, reduced word representations are unique up to interposing elements of $A$; in particular, this means that if we fix $A$-coset representatives for all but one term in a reduced word then that representation is unique. We pick $L$-minimizing $A$-coset representatives which ensure that we can restrict all of our analysis to group elements with controlled length.

Now the strategy is to split into cases depending on how much cancellation occurs in the reduced words, as in \cite{jolissaint1990rapidly,haagerup1978example}. In each case, we define families $(\varphi_i),(\psi_i)$ of auxiliary functions on either $A,$ $G$, or $H$. The $(\varphi_i)$ are defined based on values of $\varphi$ using disjoint subsets such that we have $\sum_i \|\varphi_i\|_2^2 = \|\varphi\|_2^2$. We do something similar for $\psi$. Using that the support of $\varphi$ is controlled in $L$, we can control the support of $\varphi_i$ in terms of $L$ as well. The goal is to estimate $\|(\varphi*\psi)\chi_{\Lambda_m}\|_2$, the 2-norm of the convolution restricted to reduced words with $m$ syllables. We first estimate $(\varphi*\psi)(s)$ in terms of a convolution $\varphi_i*\psi_j(k)$, where $i,j$, and $k$ are uniquely determined by $s$ and $k$ is in either $A,$ $G,$ or $H.$ Then we sum over unique representations of allowed values of $s$ to get an estimate in terms of a sum of squared 2-norms of convolutions of $\varphi_i*\psi_j$. We then apply RD/SD for $A,$ $G,$ or $H$ and add the squared 2-norms back up to conclude.

To conclude the introduction we include certain open considerations of interest. A natural question we are unable to address, in light of Theorem \ref{thm: sapir-sd-not-rd} is if there is a finitely generated group $G$ with no super-polynomial growth amenable subgroups which has SD with respect to the word length, but not RD. It would also be interesting if finitely presented such examples exist. Secondly, regarding selflessness, it is proved in \cite{amrutam2025strictcomparisonreducedgroup} that any countable selfless group (which is an effective mixed identity freeness property of a group) with RD satisfies that the reduced group $C^*$-algebra is selfless in the sense of Robert. It is very natural to now ask if one can replace RD with SD in the above. As we explained above, the proof in \cite{amrutam2025strictcomparisonreducedgroup}, cannot be generalized to this case as such. Regarding general amalgamated free products, it is a natural question if there can be any systematic or specific instances where RD/SD is still preserved, despite having high distortion.

\subsection*{Acknowledgements}
We are grateful to F. Flores for initial discussions. We also thank B. Hayes, A. Ioana, N. Ozawa, and I. Vigdorovich for useful exchanges. We are indebted to I. Chatterji for kindly alerting us to her elegant paper with Gautero, and for her very helpful feedback and generosity. This project was partially supported by the NSF grant DMS 2350049 (Kunnawalkam Elayavalli). The second author was supported in part by the Engineering and Physical Sciences Research Council (UK), grant EP/X026647/1. We acknowledge Art of Espresso, the San Diego Torrey Pines Gliderport, and Little Black Mountain, for inspiring us in unique ways towards the completion of this paper. 

\subsection*{Open Access and Data Statement} For the purpose of Open Access, the authors have applied a CC BY public copyright license to any Author Accepted Manuscript (AAM) version arising from this submission. Data sharing is not applicable to this article as no new data were created or analyzed in this work.

\section{Subexponential decay (SD)}
\subsection{Preliminary remarks}
\begin{defn}
        We say a function $f:\mathbb{N}\to\mathbb{R}_{\geq 0}$ is subexponential if $\lim f(n)^{1/n} = 1$.
\end{defn}

\begin{rem}
    We note that the following notions are equivalent for functions $f:\bN\to\bR_{\ge 0}$ which tend to infinity:
    \begin{enumerate}
        \item[(a)] $f \in o(\exp(\ee n))$ for all $\ee>0$;
        \item[(b)] $\ln(f)\in o(n)$;
        \item[(c)] $f(n)^{1/n}\to 1$.
    \end{enumerate}
    
\end{rem}

\begin{proof}
    (b) $\iff$ (c) is clear since (b) is equivalent to $\frac{1}{n}\ln(f) \to 0$ which by continuity of $\exp$ and $\ln$ is equivalent to (c).

    To see (b) $\implies$ (a), fix $\ee>0$. For all $n$ sufficiently large, $\ln(f) < \frac{\ee n}{2}$. Then $\frac{f(n)}{\exp(\ee n)} = \exp(\ln(f) - \ee n) < \exp(-\frac{\ee n}{2})$ which clearly tends to 0.

    Conversely, assume (a). Fix $\delta>0$ and choose $\ee>0$ such that $\exp(\ee) < 1+\delta$. For $n$ sufficiently large, $\frac{f(n)}{\exp(\ee n)} < 1$. Therefore $\frac{f(n)^{1/n}}{\exp(\ee)} < 1$, implying $f(n)^{1/n} < \exp(\ee) < 1+\delta$. Since $f(n)$ tends to infinity, $f(n)^{1/n}$ is eventually bounded below by 1. Hence $f(n)^{1/n}\to 1.$ 
\end{proof}

\begin{defn}
    A length function $l:G\to \bR_{\ge0}$ is a function such that for all $g,h\in G$:
    \begin{itemize}
        \item $l(e) = 0$;
        \item $l(g)=l(g^{-1})$;
        \item $l(gh)\le l(g)+l(h)$.
    \end{itemize}
\end{defn}

\begin{defn}
\label{def-SD}
    A countable group $G$ has \emph{subexponential decay} (or \emph{Property} SD) with respect to a length function $l$ if there exists a subexponential function $f$ such that for all $x\in \mathbb{C}[G]$ supported on the $l$-ball of radius $D$, one has $$\| x\|\leq f(D) \|x\|_2.$$ We say that $G$ has SD if it has SD with respect to some length function $l$. Additionally, if $\| x\|\leq f(D) \|x\|_2$ as above for all $x\in \bC[G]$, we say that $f$ is a decay function for $(G,l)$.
\end{defn}

\begin{rem}
    We note that if $G$ is finitely generated, then the word length dominates every length function on $G$ (c.f. Lemma 1.1.4(2) in \cite{jolissaint1990rapidly}). In this case, $G$ has SD if and only if $G$ has SD with respect to word length.
\end{rem}

\begin{rem}
    Let $(G,S)$ be a finitely generated group. We remark that if $\gamma(n)$ denotes the size of the radius $n$ ball of $G$ with respect to $S$, then $\lim_n \frac{\ln(\gamma(n))}{n}$ always exists. Indeed, this follows from the fact that $\gamma(n)$ is an increasing function and that $\gamma(kn)\leq \gamma(n)^k$ for all $k,n$. See Exercise 1.6 of \cite{grigorchuk2008groups}. Therefore, every group either has exponential growth or subexponential growth, depending on whether this limit is positive or zero, respectively.
\end{rem}

\begin{rem}
    There are (necessarily amenable) groups which have arbitrarily large but still subexponential growth \cite{erschler2005degrees}. In fact, the growth profiles of subexponential growth groups can be very wild, see \cite{kassabov2013groups}. Grigorchuk was the first to find subexponential groups with superpolynomial growth, see \cite{grigorchuk1985degrees}. In the same work he also discovered groups where $\limsup_n \gamma(n)$ was bounded below by an arbitrary subexponential function.

    Though not important for our purposes, we point the reader to the Gap Conjecture of Grigorchuk: that every group has either polynomial growth or growth bounded below by $a\exp(\sqrt{bn})$ for some constants $a,b>0$; see \cite{grigorchuk2014gap}. 
\end{rem}

\begin{prop}
\label{prop: amen-subexp-growth-is-SD}
    A countable amenable group $G$ with finite generating set $X$ has Property SD if and only if it has subexponential growth with respect to $X$. 
\end{prop}

\begin{proof}
    Let $\varphi = \sum_{g\in F}t_g\lambda(g) \in \bC[G]$. We first claim that if $t_g\geq 0$ for all $g\in F$, then $\|\varphi\|_1 = \|\varphi\|.$ Indeed, $\|\varphi\|\leq \|\varphi\|_1$ is clear from the triangle inequality. Conversely, if $\xi \in\ell^2(G)$ is a unit vector such that $\|\lambda(g)\xi-\xi\|_2 < \ee$ for all $g\in F$, then $\|\varphi\xi - \|\varphi\|_1\xi\|_2 < \|\varphi\|_1\ee$, implying that $\|\varphi\| \geq \|\varphi\|_1$.

    Now consider an arbitrary $\varphi = \sum_{g\in F}t_g\lambda(g) \in\bC[G]$. We can write $\varphi = \varphi_1 - \varphi_2 + i\varphi_3 - i\varphi_4$ where each $\varphi_j$ has only positive coefficients; furthermore, we can assume $\varphi_1$ and $\varphi_2$ have disjoint support and likewise for $\varphi_3$ and $\varphi_4$. Therefore $\|\varphi_1\|_1 + \|\varphi_2\|_1 = \|\Re(\varphi)\|_1 \leq \|\varphi\|_1$ and  $\|\varphi_3\|_1 + \|\varphi_4\|_1 = \|\Im(\varphi)\|_1 \leq \|\varphi\|_1$. By the previous paragraph, we have that $\|\varphi\| \leq \sum_{j=1}^4\|\varphi_j\| = \sum_{j=1}^4\|\varphi_j\|_1 \leq 2\|\varphi\|_1. $ By the Cauchy--Schwarz inequality, we get $\|\varphi\| \leq 2\sqrt{|F|}\|\varphi\|_2$.

    If $(G,X)$ has subexponential growth,  when $\varphi$ is supported on the ball of radius $n$ with respect to $X$, namely $B(n)$, we have $\|\varphi\| \leq 2\sqrt{|B(n)|}\|\varphi\|_2$. But clearly $2\sqrt{|B(n)|}$ is subexponential if $|B(n)|$ is, so $G$ has SD.

    Conversely, if $G$ does not have subexponential growth, then $|B(n)|$ is not a subexponential function, and so neither is $\sqrt{|B(n)|}$. The indicator function on $B(n)$ witnesses the lack of SD in this case.
\end{proof}

\begin{defn}
  Let $G$ be a group with a length function $l$. For $\phi\in\mathbb{C}[G]$, we define $L_l(\phi)$ as the maximum length (with respect to $l$) of any element in the support of $\phi$. If the length function is not ambiguous, we will simply refer to $L(\phi)$. For an integer $k,$ we define $C_k = \{g\in G : l(g) = k\}$ and $\chi_k$ to be the indicator function on $C_k$.
\end{defn}

\subsection{Some preservation estimates}
The results in this subsection are recorded primarily for bookkeeping purposes. The following is a useful estimate from Jolissaint's work \cite{jolissaint1990rapidly} (see Proposition 1.2.6 ((4) implies (1)) therein), which we record with proof for the reader's convenience.

\begin{prop}(\cite[Proposition~1.2.6]{jolissaint1990rapidly})
\label{prop:jolissaint_1.2.6}
Let $G$ be a countable group with length function $l_G$. Suppose $f$ is a subexponential function such that for all $k,l,m\in\mathbb{N}$, if $\phi,\psi\in\mathbb{C}[G]$ are supported on $C_k$ and $C_l$ respectively, then $$\|(\phi*\psi)\chi_m\|_2\le f(k)\|\phi\|_2\|\psi\|_2\hspace{0.5cm}\text{if }|k-l|\le m\le k+l$$ and $\|(\phi*\psi)\chi_m\|_2=0$ for other values of $m$. Then $G$ has Property SD with respect to $l_G.$

\end{prop}

\begin{proof}
Let $\phi,\psi\in\mathbb{C}[G]$ and suppose first that $\phi$ is supported on $C_k$. Define $\psi_l=\psi\cdot\chi_l$ for every $l\in\mathbb{N}$. Then one has
\begin{equation*}
\begin{aligned}
\|(\phi*\psi)\chi_m\|_2&\le\sum_{l\ge0}\|(\phi*\psi_l)\chi_m\|_2&\\
&\le f(k)\|\phi\|_2\sum_{l=|m-k|}^{m+k}\|\psi_l\|_2&\\
&\le f(k)\|\phi\|_2\sum_{l=0}^{2\min(k,m)}\|\psi_{m+k-l}\|_2 &\\ 
&\le f(k)\|\phi\|_2(2k+1)^{1/2}\left(\sum_{l=0}^{2\min(k,m)}\|\psi_{m+k-l}\|_2^2\right)^{1/2}. &\\ 
\end{aligned}
\end{equation*}
Hence, 
\begin{equation*}
\begin{aligned}
\|\phi*\psi\|_2^2&=\sum_{m\ge0}\|(\phi*\psi)\chi_m\|_2^2&\\
&\le f(k)^2(2k+1)\|\phi\|_2^2\sum_{m\ge0}\left(\sum_{l=0}^{2\min(k,m)}\|\psi_{m+k-l}\|_2^2\right) &\\
&\le f(k)^2(2k+1)^2\|\phi\|_2^2\|\psi\|_2^2.&\\
\end{aligned}
\end{equation*}
If $\phi$ is arbitrary, then $\phi=\sum\phi_k$ where $\phi_k=\phi\cdot\chi_k$. Then one has
\begin{equation*}
\begin{aligned}
\|\phi\|&\le\sum_{k\ge0}\|\phi_k\|\le\sum_{k\ge0}f(k)(2k+1)\|\phi_k\|_2&\\
&\le(L(\phi)+1)^{1/2}\left(\sum_{k\ge0}f(k)^2(2k+1)^2\|\phi_k\|_2^2\right)^{1/2}&\\
&\le(L(\phi)+1)^{1/2}(2L(\phi)+1)f(L(\phi))\left(\sum_{k\ge0}\|\phi_k\|_2^2\right)^{1/2}&\\
&\le(2L(\phi)+1)^{3/2}f(L(\phi))\|\phi\|_2.&\\
\end{aligned}
\end{equation*}
Since $g(x)=(2x+1)^{3/2}f(x)$ is subexponential, $G$ has Property SD.
\end{proof}

\begin{lem}
\label{direct_prod_SD}
If $G$ and $H$ have Property SD with decay functions $f_G$ and $f_H$ respectively, then $G\times H$ has Property SD with decay function $f(x)=f_G(x)f_H(x)$.
\end{lem}\
\begin{proof}
Let $l_G$ and $l_H$ be length functions on $G$ and $H$, respectively, such that $G$ has Property SD with respect to $l_G$, and $H$ has Property SD with respect to $l_H$. Define a length function $l$ on $G\times H$ as $l(g,h)=l_G(g)+l_H(h)$.

Given $\phi\in\mathbb{C}[G\times H]$, we define $\phi_h\in\mathbb{C}[G]$ by $\phi_h(g)=\phi(g,h)$. For $\phi,\psi\in\mathbb{C}[G\times H]$, we have

\begin{equation*}
\begin{aligned}
\|\phi*\psi\|_2^2&=\sum_{g\in G}\sum_{h\in H}|\phi*\psi(g,h)|^2&\\
&=\sum_{g\in G}\sum_{h\in H}\left|\sum_{g'\in G}\sum_{h'\in H}\phi(g',h')\psi(g'^{-1}g,h'^{-1}h)\right|^2.&\\
\end{aligned}
\end{equation*}
Since $\sum_{g'\in G}\phi(g',h')\psi(g'^{-1}g,h'^{-1}h)=\phi_{h'}*\psi_{h'^{-1}h}(g)$, we have

\begin{equation*}
\|\phi*\psi\|_2^2=\sum_{g\in G}\sum_{h\in H}\left|\sum_{h'\in H}(\phi_{h'}*\psi_{h'^{-1}h})(g)\right|^2\le\sum_{h\in H}\left(\sum_{h'\in H}\left\|\phi_{h'}*\psi_{h'^{-1}h}\right\|_2\right)^2
\end{equation*}
where the last inequality follows from the definition of the $2$-norm and the triangle inequality. We can now use the fact that $G$ has SD with respect to ${l_G}$ to write

\begin{equation*}
\begin{aligned}
\|\phi*\psi\|_2^2&\le\sum_{h\in H}\left(\sum_{h'\in H} f_G(L_{l_G}(\phi_{h'}))\|\phi_{h'}\|_2\|\psi_{h'^{-1}h}\|_2\right)^2&\\
&\le f_G(L_l(\phi))^2\sum_{h\in H}\left(\sum_{h'\in H}\|\phi_{h'}\|_2\|\psi_{h'^{-1}h}\|_2\right)^2.&\\
\end{aligned}
\end{equation*}
We now define $\tilde{\phi},\tilde{\psi}\in\mathbb{C}[H]$ by $\tilde{\phi}(h)=\|\phi_{h}\|_2$ and $\tilde{\psi}(h)=\|\psi_{h}\|_2$. We can use the fact that $\sum_{h'\in H}\|\phi_{h'}\|_2\|\psi_{h'^{-1}h}\|_2=(\tilde{\phi}*\tilde{\psi})(h)$ and that $H$ has SD with respect to $l_H$ to write

\begin{equation*}
\begin{aligned}
\|\phi*\psi\|_2^2&\le f_G(L_l(\phi))^2\sum_{h\in H}|(\tilde{\phi}*\tilde{\psi})(h)|^2&\\
&=f_G(L_l(\phi))^2\|\tilde{\phi}*\tilde{\psi}\|_2^2&\\
&\le f_G(L_l(\phi))^2f_H(L_l(\phi))^2\|\tilde{\phi}\|_2^2\|\tilde{\psi}\|_2^2.&\\
\end{aligned}
\end{equation*}
We note that $$\|\tilde{\phi}\|_2=\left(\sum_{h\in H}\|\phi_{h}\|_2^2\right)^{1/2}=\left(\sum_{h\in H}\sum_{g\in G}|\phi(g,h)|^2\right)^{1/2}=\|\phi\|_2$$ and similarly for $\psi$. Hence,

\begin{equation*}
\begin{aligned}
 \|\phi*\psi\|_2&\le f_G(L_l(\phi))f_H(L_l(\phi))\|\tilde{\phi}\|_2\|\tilde{\psi}\|_2=f_G(L_l(\phi))f_H(L_l(\phi))\|\phi\|_2\|\psi\|_2&\\ 
\end{aligned}
\end{equation*}
which proves the lemma.
\end{proof}
The following is pretty much contained in Jolissaint's work \cite{jolissaint1990rapidly}, but for the convenience of the reader, we place the argument below, keeping track of the constants. 

\begin{lem}
\label{free_prod_SD}
If $G$ and $H$ have Property SD with decay functions $f_G$ and $f_H$ respectively, then $G*H$ has Property SD with decay function 
$$f(x)=(2x+1)^{7/2}\max(f_G(x),f_H(x)).$$
\end{lem}

\begin{proof}
Suppose $G$ and $H$ have Property SD with respect to the length functions $l_G$ and $l_H$, respectively. Define a length function $l_{G*H}$ on $G*H$ by $$l_{G*H}(g_1h_1...g_nh_n)=\sum_{i=1}^n(l_G(g_i)+l_H(h_i))\hspace{0.5cm}(g_1h_1...g_nh_n\text{ a reduced word}).$$
We also define
$$\Lambda_m=\{g\in G*H|g=s_1...s_m\text{ as a reduced word}\}.$$
Following arguments laid out in the proof of \cite[Lemma 2.2.4]{jolissaint1990rapidly}, one can show that for all $k,l,m\in\mathbb{N}$ and $\phi,\psi\in\mathbb{C}[G*H]$ with support in $\Lambda_{k}$ and $\Lambda_{l}$ respectively, 
$$\|(\phi*\psi)\chi_{\Lambda_m}\|_2\le\max(f_G(L_{l_{G*H}}(\phi)),f_H(L_{l_{G*H}}(\phi)))\|\phi\|_2\|\psi\|_2$$
where $\chi_{\Lambda_m}$ is the indicator function on $\Lambda_m$.

It now suffices to show that $G*H$ satisfies the hypotheses of Proposition \ref{prop:jolissaint_1.2.6}. Suppose $k,l,m\in\mathbb{N}$ and $\phi,\psi\in\mathbb{C}[G*H]$ are supported on $C_k$ and $C_l$ respectively. From the definition of $l_{G*H}$ it follows that $C_k\subset\bigcup_{i=0}^k\Lambda_k$ so 
$$\phi=\sum_{i=0}^k\phi_i\hspace{1cm}\text{and}\hspace{1cm}\psi=\sum_{j=0}^l\psi_j$$
where $\phi_i=\phi\cdot\chi_{\Lambda_i}$ and $\psi_j=\psi\cdot\chi_{\Lambda_j}$. Set $g(x)=\max(f_G(x),f_H(x))$ and fix $j\in\{0,...,k\}$. We have 
\begin{equation*}
\begin{aligned}
 \|(\phi_j*\psi)\chi_{m}\|_2^2&\le\sum_{i=0}^m\|(\phi_j*\psi)\chi_{\Lambda_i}\|_2^2&\\
 &\le g(k)^2\|\phi_j\|_2^2\sum_{i=0}^m\left(\sum_{k=|j-i|}^{j+i}\|\psi_k\|_2\right)^2&\\
&\le(2j+1)g(k)^2\|\phi_j\|_2^2\sum_{i=0}^m\sum_{k=0}^{2\min(j,i)}\|\psi_{j+i-k}\|_2^2&\\
 &\le(2j+1)^2g(k)^2\|\phi_j\|_2^2\|\psi\|_2^2.
 \end{aligned}
\end{equation*}
Hence,
\begin{equation*}
\begin{aligned}
    \|(\phi*\psi)\chi_m\|_2&\le\sum_{j=0}^k\|(\phi_j*\psi)\chi_m\|_2&\\
    &\le g(k)\|\psi\|_2\sum_{j=0}^k(2j+1)\|\phi_j\|_2&\\
    &\le(2k+1)(k+1)g(k)\|\phi\|_2\|\psi\|_2&\\
    &\le(2k+1)^{2}g(k)\|\phi\|_2\|\psi\|_2
 \end{aligned}
\end{equation*}
for $k\ge1$. Therefore, by Proposition \ref{prop:jolissaint_1.2.6}, $G*H$ has property SD with decay function $f(x)=(2x+1)^{7/2}g(x)$.
\end{proof}

\subsection{An SD group without RD and superpolynomial amenable subgroups}

The following result adapts the construction of Sapir from \cite{sapir2015rapid} to construct a group, albeit not finitely generated, that has SD but not RD, and all of its amenable subgroups are of polynomial growth.

\begin{thm}
\label{thm: sapir-sd-not-rd}
Consider the group $G=*_{n\ge1}\mathbb{Z}^{n }$ together with the length function $l$ defined as follows. Fix $f:\bN\to\bN$ a strictly increasing function such that $f(1) = 1$ and $f^{-1}(N)$ is $o\left(\frac{N}{\ln(N)}\right)$. Let $a_{1,n},...,a_{n,{n}}$ denote the standard generators of $\mathbb{Z}^{n}$; we set $l(a_{i,n})=f(n)$ for all $a_{i,n}$. Then $G$ has SD (with respect to $l$) but not RD.
\end{thm}

\begin{proof}
By Lemma \ref{direct_prod_SD}, the decay function with respect to word length $l_w$ for $\mathbb{Z}^{n}$ is $g_{l_w,n}(k)=(2k+1)^{n/2}$, so the decay function with respect to $l$ is 
$$g_{l,n}(k)=\left(\frac{2k}{f(n)}+1\right)^{n/2}.$$
Therefore by Lemma \ref{free_prod_SD}, the decay function for $\mathbb{Z}^{{1}}*\dotsm *\mathbb{Z}^{n}$ is 
\begin{align*}
    g_n(k)&=(2k+1)^{7(n-1)/2}\max(g_{l,1}(k),g_{l,2}(k),...,g_{l,n}(k))\\
    &\le(2k+1)^{7(n-1)/2}\left(2k+1\right)^{n/2} \\
    &\le (2k+1)^{4n}.
\end{align*}
If $\phi\in\mathbb{C}[G]$, then $\phi\in\mathbb{C}\left[\mathbb{Z}^{1}*\dotsm *\mathbb{Z}^{f^{-1}(L_l(\phi))}\right]$. Hence, denoting $N = L_l(\phi)$, we have
\begin{align*}
    \|\phi\|&\le g_{f^{-1}(N)}(N)\|\phi\|_2 \\
    &\le (2N+1)^{4f^{-1}(N)}\|\phi\|_2\\
    &\leq \exp(4f^{-1}(N)\ln(3N))\|\phi\|_2
\end{align*}

which, in combination with the assumption that $f^{-1}(N)$ is $o\left(\frac{N}{\ln(N)}\right)$, shows $G$ has property SD.

To show $G$ does not have RD, let $l'$ be an arbitrary length function on $G$, and suppose for the purpose of contradiction that that there exists a polynomial $P(k)$ of degree $N$ that is a decay function for $G$. Define $\phi_n\in\mathbb{C}\left[\mathbb{Z}^{2(N+1)}\right]$ as $$\phi_n=\sum_{a\in B_n}a$$
where $B_n$ is the ball of radius $n$ in $\mathbb{Z}^{2(N+1)}$ with respect to word length. Define $\xi_m=\frac{\phi_{m}}{\|\phi_m\|_2}$. Since balls of increasing radius are F{\o}lner sets in $\mathbb{Z}^{2(N+1)}$, for every $\epsilon>0$ we can find $m$ large enough such that 
$$\|a*\xi_m-\xi_m\|_2<\frac{\epsilon}{|B_n|}$$ 
for all $a\in B_n$. Since $\phi_n*\xi_m=\sum_{a\in B_n}a*\xi_m$, it follows that \
$$\|\phi_{n}*\xi_m-|B_n|\xi_m\|_2\le\sum_{a\in B_n}\|a*\xi_m-\xi_m\|_2<\epsilon.$$
Hence, $\|\phi_n*\xi_m\|_2\to|B_n|$ as $m\to\infty$. In particular, $\|\phi_n\|\ge|B_n|$, so 
$$\frac{\|\phi_n\|}{\|\phi_n\|_2}\ge\sqrt{|B_n|}\sim(2n+1)^{N+1}.$$
If we consider $\phi_n$ as an element of $G$, then $L_{l'}(\phi_n)\le n\max_i l'(a_i)$ where maximum is taken over the standard generators of $\bZ^{2(N+1)}$. Therefore,
$$\|\phi_n\|\ge\left(\frac{L_{l'}(\phi_n)}{\max_il'(a_i)}+1\right)^{N+1}\|\phi_n\|_2$$ 
which contradicts our assumption that the decay function for $(G,l')$ is a polynomial of degree $N$. Hence, $G$ does not have property RD.
\end{proof}

\subsection{Graph product estimates}
\label{sect: graph-prod-SD}
The following is immediate from the work of Ciobanu--Holt--Rees \cite{graphprodRD}, but for the convenience of the reader, we extract here their proof, making the necessary adaptations that are needed for SD. One of the reasons we record this proof (despite the fact that the result can be recovered from our Theorem \ref{thm: no distortion}), is because the precise decay function estimates are sharper.

\begin{thm}
\label{thm:graph-prod-sd}
    Suppose that $G=G(\Gamma,G_v,v\in V)$ is a graph product of groups over a finite simplicial graph $\Gamma=(V,E)$, and suppose that each vertex group $G_v$ satisfies SD. Then $G$ satisfies SD.
\end{thm}
\begin{proof}
Following the notation in \cite{graphprodRD}, let $\mathcal{K}$ be the set of cliques of $\Gamma$ and $\mathcal{K}_m$ be the cliques of size $m$. Given a subset $J$ of $V$, define $G_J$ as the subgroup generated by the elements of the subgroups $G_v$ where $v\in J$. 
Every $g\in G$ can be expressed as a product $y_1...y_k$ with each $y_i$ in a vertex group $G_{v_i}$. We will refer to such a representation as an \textit{expression}. We will call the $y_i$ the \textit{syllables} of the expression and say that the expression has \textit{syllable length k}. We define a \textit{syllable length} of $g$, denoted $\lambda(g)$, as the minimum number of syllables needed to express $g$. We say that an expression for $g$ is \textit{reduced} if it has syllable length $\lambda(g)$. Analogous to the proof of Lemma \ref{free_prod_SD}, we define $\Lambda_k=\{g\in G|\lambda(g)=k\}$.

For each $v\in V$, let $l_v$ be a length function on $G_v$ such that $G_v$ has property SD with respect to $l_v$. Given $g\in G$ and a reduced expression $g=y_1...y_k$ with $y_j\in G_{v_j}$ we define a length function $l_G$ on $G$ by 
$$l_G(g)=\sum_{i=1}^kl_{v_i}(y_i).$$ 
By Lemma \ref{direct_prod_SD}, if $J\in\mathcal{K}$, then $G_J$ has SD with respect to $l_G$.

Let $\chi_{k}$ be the characteristic function on $C_k$ and $\chi_{(k)}$ the characteristic function on $\Lambda_k$. We define $\phi_{k}$ as the product $\phi\cdot\chi_{k}$ and $\phi_{(k)}$ as the product $\phi\cdot\chi_{(k)}$. 

We now show the equivalent of Proposition 4.3 in \cite{graphprodRD}:
\begin{prop}
\label{prop: graph_prod}
    There exists a subexponential function $g(x)$ such that for all $\phi,\psi\in\mathbb{C}[G]$, $k,l,m\in\mathbb{N}$, and $|k-l|\le m\le k+l$,
    $$\|(\phi_{(k)}*\psi_{(l)})_{(m)}\|_2\le g(L_{l_G}(\phi_{(k)}))\|\phi_{(k)}\|_2\|\psi_{(l)}\|_2.$$
\end{prop}

\begin{proof}
Following the definitions in Section 5 of \cite{graphprodRD}; given $\phi_{(k)},\psi_{(l)}$ and $p\ge0$, we define $\phi_{(k-p)}^{(p)}$ and $^{(p)}\psi_{(l-p)}$ by
$$\phi_{(k-p)}^{(p)}(u)=\begin{cases}
 \sqrt{\sum_{w\in\Lambda_p}|\phi_{(k)}(uw)|^2}&\text{ if }u\in\Lambda_{k-p},\\
 0&\text{ otherwise};
\end{cases}$$
$$^{(p)}\psi_{(l-p)}(u)=\begin{cases}
 \sqrt{\sum_{w\in\Lambda_p}|\psi_{(l)}(w^{-1}u)|^2}&\text{ if }u\in\Lambda_{l-p},\\
 0&\text{ otherwise}.
\end{cases}$$
Suppose $m=k+l-q$ with $q\ge0$. Given $g_1\in\Lambda_{k-q+p}$, $g_2\in\Lambda_{l-q+p}$, and $p\ge0$, define $^{g_1}\phi^{(p)}_{(q-2p)}$ and $^{(p)}\psi_{(q-2p)}^{g_2}$ by
$$^{g_1}\phi^{(p)}_{(q-2p)}(v)=
 \phi_{(k-p)}^{(p)}(vg_1)\hspace{1cm}^{(p)}\psi^{g_2}_{(q-2p)}=  \prescript{(p)}{}\psi_{(l-p)}(g_2v).$$
Following the arguments in Section 6 of \cite{graphprodRD}, one can show that there exists a polynomial $Q(k)$ such that
$$\|(\phi_{(k)}*\psi_{(l)})_{(m)}\|_2^2\le Q(k)\sum_{p=1}^{p=\lfloor q/2\rfloor}\sum_{J\in\mathcal{K}_{q-2p}}\sum_{\substack{g_1\in\Lambda_{k-q+p} \\ g_2\in\Lambda_{l-q+p}}}\|^{g_1}\phi^{(p)}_{(q-2p)}*_J\prescript{(p)}{}\psi_{(q-2p)}^{g_2}\|_{2;J}^2$$
where $*_J$ is the convolution is over $G_J$ and $\|\cdot\|_{2;J}$ is the 2-norm in $G_J$. Since SD holds with respect to $l_G$ in each of the $G_J$, we have
$$\|^{g_1}\phi^{(p)}_{(q-2p)}*_J\prescript{(p)}{}\psi_{(q-2p)}^{g_2}\|_{2;J}^2\le f_J\left(L_{l_G;J}\left(^{g_1}\phi_{(q-2p)}^{(p)}\right)\right)^2\|^{g_1}\phi_{(q-2p)}^{(p)}\|_{2;J}^2\|^{(p)}\psi_{(q-2p)}^{g_2}\|^2_{2;J}$$ where $f_J$ is the decay function on $G_J$ and $L_{l_G;J}(\phi)$ is the maximum $l_G$-length of any element in the support of $\phi$ restricted to $J$.

Note that if $g_1\in\Lambda_{k-q+p},|J|=q-2p$, and some element $v\in J$ is in the support of $^{g_1}\phi_{(q-2p)}^{(p)}$, then there exists $w\in\Lambda_p$ such that $vg_1w$ is in the support of $\phi_{(k)}$. Since $\lambda(v)\le|J|=q-2p$ and $\lambda(g_1)=k-q+p$, this implies $\lambda(v)=q-2p$. In particular, $\lambda(vg_1w)=\lambda(v)+\lambda(g_1)+\lambda(w)$, so $l_G(v)\le l_G(vg_1w)$. It follows that $L_{l_G;J}\left(^{g_1}\phi_{(q-2p)}^{(p)}\right)\le L_{l_G}(\phi_{(k)})$.

Define $f$ as $f(x)=\max_{J\in\mathcal{K}}(f_J(x))$. As there are finitely many cliques, $f$ is subexponential. Once again, following arguments in Section 6 of \cite{graphprodRD} and using the fact that $L_{l_G;J}\left(^{g_1}\phi_{(q-2p)}^{(p)}\right)\le L_{l_G}(\phi_{(k)})$ for all $0\le p\le\lfloor q/2\rfloor$, $g_1\in \Lambda_{k-q+p}$, and $J\in\mathcal{K}_{q-2p}$, one can show that there exists a polynomial $P$ such that
$$\|(\phi_{(k)}*\psi_{(l)})_{(m)}\|_2^2\le P(L_{l_G}(\phi_{(k)}))f(L_l(\phi_{(k)}))^2\|\phi_{(k)}\|_2^2.$$ 
Letting $g(x)=\sqrt{P(x)}f(x)$ completes the proof of Proposition \ref{prop: graph_prod}.
\end{proof}

We now complete the proof of Theorem \ref{thm:graph-prod-sd} by showing that $G$ satisfies the hypotheses of Proposition \ref{prop:jolissaint_1.2.6}. Suppose $\phi,\psi\in\mathbb{C}[G]$ are supported on $C_k$ and $C_l$, respectively, and $|k-l|\le m\le k+l$. Using arguments nearly identical to those in Section 4 of \cite{graphprodRD} together with Proposition \ref{prop: graph_prod} one can show that for a fixed $j$,
\begin{equation*}
\begin{aligned}
\|(\phi_{(j)}*\psi)_m\|_2^2&\le(2j+1)g(L(\phi_{(k)}))^2\|\phi_{(j)}\|_2^2\sum_{p=0}^m\sum_{i=|j-p|}^{j+p}\|\psi_{(i)}\|&\\
&\le(2j+1)^2g(L(\phi_{(k)}))^2\|\phi_{(j)}\|_2^2\|\psi\|_2^2.
\end{aligned}
\end{equation*}
We now have
\begin{equation*}
\begin{aligned}
    \|(\phi*\psi)_m\|_2^2&\le\left(\sum_{j=0}^k\|(\phi_{(j)}*\psi)_m\|\right)^2&\\
    &\le(k+1)\sum_{j=0}^k\|(\phi_{(j)}*\psi)_m\|^2&\\ &\le(k+1)\sum_{j=0}^k(2j+1)^2g(k)^2\|\phi_{(j)}\|_2^2\|\psi\|_2^2&\\
    &\le(k+1)(2k+1)^2g(k)^2\|\phi\|_2^2\|\psi\|_2^2.&\\
\end{aligned}
\end{equation*}
Hence, $G$ satisfies the hypotheses of Proposition \ref{prop:jolissaint_1.2.6}, which proves Theorem \ref{thm:graph-prod-sd}.
\end{proof}

\subsection{Selflessness for free products}
\label{sect: selfless}

As an application of considering property SD, we prove here selflessness for a more general family of free products than considered in \cite{amrutam2025strictcomparisonreducedgroup}. 

\begin{thm}\label{subexponential selfless free products}
    Let $G$ and $H$ be countable groups admitting property SD, and such that $G$ has a torsion free element, and $H$ is infinite. Then $C^*_r(G*H)$ is selfless. 
\end{thm}

The proof is crucially motivated by the following lemma which is a direct consequence of Lemma 4.4 in \cite{hayes2025selflessreducedfreeproduct}, which in turn is an application of the Khintchine-type inequality of Ricard-Xu \cite{KhinRX}.

\begin{lem}\label{lem: RDP estimate on layers for free prod}
Suppose $G$ and $H$ are countable groups with property SD, i.e, there are subexponential functions $f_G$ and $f_H$ such that $G$ and $H$ admit $f_G$-decay and $f_H$-decay respectively. Then, for any element $x\in \mathbb{C}[G*H]$ such that $x$ is supported on words of at most $n$ many alternations in $G*H$, and each letter is at most of length $m$ inside $G,H$, we have $$\|x\|\leq 2\sqrt{2}(n+1)\max\left(f_G(m), f_H(m)\right)\|x\|_2.$$
\end{lem}

Now we are ready to prove Theorem \ref{subexponential selfless free products}.

\begin{proof}[Proof of Theorem \ref{subexponential selfless free products}]
    It suffices to prove for all $\epsilon>0$ and all $x\in \mathbb{C}[G*H*\langle t \rangle]$ such that $x$ is supported on words of length at most $n$ in $G*H*\langle t\rangle $, and each letter is at most of length $n$ inside $G*H$ (for convenience, we denote the collection of such elements as $B_n$), the existence of a homomorphism $\phi:G*H*\mathbb{Z} \to G*H$ such that $\phi|_{G*H}=Id_{G*H}$ and that $\|\phi(x)\| \leq \|x\| + \epsilon.$ Following Proposition 3.1 of \cite{amrutam2025strictcomparisonreducedgroup}, we get homomorphisms $\phi_n: G*H*\langle{t}\rangle\to G*H$ which are identity on $G*H$ and are injective on $B_n$, and such that $|\phi_n(t)|=4n+1$. 

    Fix $x\in \mathbb{C}[B_n]$ and for now we pick a large $m\in \mathbb{N}$ (satisfying $m>n$), and compute using Lemma \ref{lem: RDP estimate on layers for free prod}:   
    
    \begin{align*}
        \|\phi_{2nm}(x)\|^{2m} &= \|\phi_{2nm}(x^*x)^m\| \\
        &\leq 2\sqrt{2}(8nm)\max\left(f_G(4nm), f_H(4nm)\right)\|\phi_{2nm}(x^*x)^m\|_2 \\
        &= 2\sqrt{2}(8nm)\max\left(f_G(4nm), f_H(4nm)\right)\|(x^*x)^m\|_2\\
        &\leq 2\sqrt{2}(8nm) \max(f_g(4nm), f_H(4nm))\|x\|^{2m}.
    \end{align*}

The crucial point above is to note that the entire word length of $\phi_{2nm}(x^*x)^{m}$ is not detected in line 2, thanks to Lemma \ref{lem: RDP estimate on layers for free prod}. Instead the individual word lengths of the alternating pieces are detected, but this is bounded by $4nm$. This is thanks to Proposition 3.2 of \cite{amrutam2025strictcomparisonreducedgroup} which constructs the maps $\phi_m$ as replacing $t$ with a word of 3 alternations in $G*H$, each letter being at most length $4mn$. Additionally, the number of alternations does appear, which is bounded by $8nm$, but it is not composed with the decay function, and is instead multiplied to the estimate.

Now taking $2m^{th}$ root both sides and taking $m$ large, we obtain the result.
\end{proof}

We remark that in essence, the same proof idea can recover selflessness of free products of $C^*$-probability spaces admitting filtrations with subexponential growth in the sense of \cite{KirchbergVaillant}, following the proof structure of \cite{hayes2025selflessreducedfreeproduct}. Due to the scope of the present article, we avoid including the technical details here, and leave it to the reader as an exercise. 

\subsection{Fr\'{e}chet *-subalgebras of $C_r^*(G)$}

We show here that as in Section 1 of \cite{jolissaint1989k}, one can construct Fr\'{e}chet subalgebras of $C_r^*(G)$ by considering seminorms related to a given subexponential function. However, we are unable here to show that they are holomorphically closed and therefore cannot relate their K-theory to that of $C_r^*(G)$ (see Remark \ref{frechet-bad}).

For $\varphi\in \bC[G]$, $L$ a length function on $G$, and $f$ a subexponential function, define $$\|\varphi\|_{2,f}^2 = \sum_{g\in G} |\varphi(g)|^2 f(L(g))^2.$$ We define $H_f(G) = \overline{\bC G}^{\|\cdot\|_{2,f}}$.  Let $p_N$ denote the orthogonal projection on $\ell^2(G)$ onto the span of the group elements of length at most $N$. Define 
\begin{align*}
    h(N) &= \ln\left(\frac{N}{\ln({f}(N))}\right); \\
    j(N) &= \ln(h(N)).
\end{align*} Note that since $f$ is subexponential, $h$ tends to infinity. Without loss of generality and for simplicity, we assume $h$ is increasing. Define, for $\alpha\in(0,1),$ $f_\alpha(N) := \exp\left(\frac{N}{h(N)^\alpha}\right)$. We define seminorms as follows, for $\alpha,\beta\in (0,1)$ and $q\in\bN$:

$$\nu_{\alpha,\beta,q}(a) = \sup_{N\ge 1} \{N^qf_\beta(N)\|(1-p_N)ap_{N-\frac{N}{j(N)^\alpha}}\|+\|p_{N-\frac{N}{j(N)^\alpha}}a(1-p_N)\|\}.$$

As in Lemma 1.2 of \cite{jolissaint1989k}, we achieve the following identity for $a,b\in \bC[G]$:

$$\nu_{\alpha,\beta,q}(ab) \le \nu_{(1+\alpha)/2,\beta,q}(a)\|b\| + \nu_{(1+\alpha)/2,\beta,q}(b)\|a\|.$$

It follows that the space $T_f^\infty(G)$ defined by the closure of $\bC[G]$ with respect to $\|\cdot\|$ and the family $(\nu_{\alpha,\beta,q})$ is a *-closed Fr\'{e}chet subalgebra of $C^*_r(G)$. We may also define $H_f^\infty(G)$ as follows:
$$H_f^\infty(G) = \bigcap_{\beta\in(0,1),q\in\bN} H_{N^qf_\beta}(G).$$

From the same proof as Proposition 1.3 in \cite{jolissaint1989k} and the fact that $\exp(N/j(N))$ dominates all of the $N^qf_\beta$, we have that $H_f^\infty(G) = T_f^\infty(G)$ whenever $G$ has $f$-decay. We record this as a proposition.

\begin{prop}
    If $G$ has $f$-decay then $H_f^\infty = T_f^\infty$ as defined above is a *-closed Fr\'{e}chet subalgebra of $C_r^*(G)$.
\end{prop}

\begin{rem}
    \label{frechet-bad}
      Jolissaint's techniques break down when trying to show that $T_f^\infty$ is inverse-closed, the key step in showing that their K-theory matches $C^*_r(G)$. Indeed, the argument in Theorem 1.4 of \cite{jolissaint1989k} requires the function $j$ defined above to satisfy that $\exp(j(N))$ dominates $f$. However, in the proof that $H_f^\infty(G) = T_f^\infty(G)$ we crucially used the fact that $\exp(N/j(N))$ dominates $f$. For a polynomial, this is possible; take $j(N) = \sqrt{N}$. But for $f = \exp(\sqrt{n})$, no such $j$ exists.
\end{rem}

\begin{rem}
    Jolissaint constructs another algebra, $H^\omega(G)$, in the Appendix in \cite{jolissaint1989k}. Using similar ideas, one can again construct a *-closed Fr\'{e}chet subalgebra of $C^*_r(G)$ when $G$ has $f$-decay. However, we cannot show it is holomorphically-closed, with essentially the same obstruction as for $H_f^\infty(G).$
\end{rem}

.

\section{Preservation under amalgamated free products}

In this section we will prove the main theorems. Let $A\le G,H$ be a common subgroup and consider the amalgamated free product (AFP) $\Gamma = G*_A H$. Assume that $G,H$ admit length functions $L_G,L_H$ such that $L_G|_A = L_H|_A$; call this common restriction $L_A$. We may assume that $L_G,L_H$ are integer-valued as in \cite{jolissaint1990rapidly}.

\begin{defn}
\label{def: afp-len-agree-on-A}
    Define a function $L:\Gamma \to \bN$ as follows: 
\begin{align*}
    L(k) := \min\{&L_A(a) + L_G(g_1) + L_H(h_1) + \ldots \mid \\
    &k = ag_1h_1\cdots, \ a\in A, g_i\in G\setminus A, h_i\in H\setminus A\}
\end{align*}
\end{defn}

By symmetry, the following lemma also holds for $h\in H$ and $L_H$.
\begin{lem}
\label{lem: L = L_G}
    Let $g\in G$. Then $L(g) = L_G(g)$.
\end{lem}

\begin{proof}
    First assume $g\in A$. Then by definition $L(g) = L_A(g) = L_G(g)$. Now assume $g\in G\setminus A$. Then for all $a\in A,$ $L_G(g) \leq L_G(a) + L_G(a^{-1}g) = L_A(a) + L_G(a^{-1}g)$. But minimizing the the right hand side over all $a\in A$ is exactly $L(g)$, so that $L_G(g) \leq L(g)$. $L(g) \le L_G(g)$ is clear by the definition. 
\end{proof}

The previous lemma implies that really, we can define $L(a) = L_A(a)$ for $a\in A$ and $L(k) = \min\{L_G(g_1) + L_H(h_1) + \ldots \mid k = g_1h_1\cdots, \ g_i\in G\setminus A, h_i\in H\setminus A\}$ for $k\in\Gamma\setminus A$. We also remark that for $g_i\in G\setminus A$ and $h_i\in H\setminus A$,
\begin{align}
\label{afp-len-iden}
    L(g_1h_1\cdots g_nh_n) &= \min_{a_i,b_i\in A} L_G(g_1a_1) + L_H(a_1^{-1}h_1b_1) + L_G(b_1^{-1}g_2a_2) + \ldots
\end{align}

\begin{lem}
\label{lem: L-well-def}
    If $L_G,L_H$ are length functions on $G,H$ respectively that agree on $A$, then $L$ as defined above is a length function on $\Gamma = G*_AH$.
\end{lem}

\begin{proof}
    We must prove that $L(e)=0$, that $L$ is symmetric, and that $L$ is subadditive. That $L(e) = 0$ is clear. To show that $L(k) = L(k^{-1})$ for all $k\in \Gamma$, first note that if $k\in A$ then this follows from the fact that $L_A$ is symmetric. For $k\in \Gamma\setminus A$, write $k = g_1h_1\cdots g_nh_n$. By (\ref{afp-len-iden}), we have that $L(k) = L_G(g_1a_1) + L_H(a_1^{-1}h_1 a_2) + \ldots$ for some $a_1,\ldots,a_{2n}\in A$. But $k^{-1} = (h_n^{-1}a_{2n})(a_{2n}^{-1}g_na_{2n-1})\cdots (a_2^{-1}h_1a_1)(a_1^{-1}g_1)$. Therefore 
    \begin{align*}
        L(k^{-1}) &\leq L_H(h_n^{-1}a_{2n})+L_G(a_{2n}^{-1}g_na_{2n-1})+\ldots+ L_H(a_2^{-1}h_1a_1)+L_G(a_1^{-1}g_1)\\
        &= L_G(g_1a_1) + L_H(a_1^{-1}h_1 a_2) + \ldots\\
        &= L(k).
    \end{align*}
    By symmetry, we clearly have $L(k) = L(k^{-1})$.

    The proof of the subadditivity of $L$ has several cases, up to swapping the roles of $G$ and $H.$ Let $a\in A$ and $k,k'\in\Gamma\setminus A.$
    \begin{enumerate}
        \item $L(ak) \leq L(a)+L(k)$;
        \item $L(ka) \le L(k) + L(a)$;
        \item $L(kk') \le L(k) + L(k')$ where $k$ ends with an element in $G$ and $k'$ starts with an element in $H$;
        \item $L(kk')\le L(k)+L(k')$ where $k$ ends with an element in $G$ and $k'$ starts with an element in $G$ and the product of these elements is not in $A$;
        \item $L(kk')\le L(k)+L(k')$ where $k$ ends with an element in $G$ and $k'$ starts with an element in $G$ and the product of these elements is in $A$.
    \end{enumerate}
    Throughout the following arguments, we slightly abuse notation and write $L$ in place of $L_G$ and $L_H$, understanding that the restrictions of $L$ to $G$ and $H$ respectively are $L_G$ and $L_H.$
    \begin{enumerate}
        \item Write $k = g_1h_1\cdots g_nh_n$ where each $g_i\in G\setminus A$ and $h_i\in H\setminus A$. Then for all $a_i\in A$, we have
        \begin{align*}
            L(ak) &\le L(ag_1a_1) + L(a_1^{-1}h_1a_2) + \ldots \\
            &\le L(a) + L(g_1a_1) + L(a_1^{-1}h_1a_2) + \ldots 
        \end{align*}
        Using (\ref{afp-len-iden}) and minimizing over all choices of $a_i\in A$ in the final line above, we get that $L(ak) \leq L(a) + L(k)$. 
        \item Apply (1) and symmetry to see that $L(ka) = L(a^{-1}k^{-1}) \leq L(a^{-1}) + L(k^{-1}) = L(a) + L(k)$.
        \item Write $k = g_1h_1\cdots g_n$ and $k' = h_1'g_2'\cdots h_m'$. Then, by (\ref{afp-len-iden}), we have 
        \begin{align*}
            L(kk') &= \min_{a_i\in A} L(g_1a_1) + \ldots + L(a_{2n-2}^{-1}g_na_{2n-1}) + L(a_{2n-1}^{-1}h_1'a_{2n}) + \ldots \\
            &\le \min_{a_i\in A} L(g_1a_1) + \ldots + L(a_{2n-2}^{-1}g_n) + L(h_1'a_{2n}) + \ldots \\
            &= L(k) + L(k'),
        \end{align*}
        where when we add the restriction of $a_{2n-1}=e$, the minimum can only increase.
        \item Now write $k = g_1h_1\cdots g_n$ and $k' = g_1'h_1'\cdots h_m'$, where $g_ng_1'\notin A$. We have
        \begin{align*}
            L(kk') &= \min_{a_i\in A} L(g_1a_1) + \ldots + L(a_{2n-2}^{-1}g_ng_1'a_{2n-1}) + L(a_{2n-1}^{-1}h_1'a_{2n}) + \ldots \\
            &\le \min_{a_i\in A} L(g_1a_1) + \ldots + L(a_{2n-2}^{-1}g_n) + L(g_1'a_{2n-1}) + L(h_1'a_{2n}) + \ldots \\
            &= L(k) + L(k'),
        \end{align*}
        where we used the subadditivity of $L_G$ in the inequality.
        \item Again write $k = g_1h_1\cdots g_n$ and $k' = g_1'h_1'\cdots h_m'$, where now $g_ng_1'\in A$. Pick $r$ maximal such that $g_{n-r+1}h_{n-r+1}\cdots g_ng_1'h_2'\cdots h'_{r-1}g'_r$ is in $A$. In particular, this implies that $h_{n-s+1}\cdots g_ng_1'\cdots h_{s-1}$ and $g_{n-s+1}\cdots g_ng_1'\cdots g_s$ are in $A$ for all $s\leq r$. Let $a_i,a_i' \in A$ be some collection of elements. Then
        \begin{align*}
        \label{case-five}
            L(kk') &\le L(g_1a_1) + \ldots + L(a_{2(n-r)-2}^{-1}g_{n-r}a_{2(n-r)-1}) \\ &+L(a_{2(n-r)-1}^{-1}h_{n-r}g_{n-r+1}h_{n-r+1}\cdots g_ng_1'h_2'\cdots h'_{r-1}g'_rh'_r a'_{2r}) \\
            &+ L(a_{2r}^{'-1}g'_{r+1}a_{2r+1}) + \ldots 
        \end{align*}
    
    Let us now focus on the complicated middle term. \begin{align*}
        &L(a_{2(n-r)-1}^{-1}h_{n-r}g_{n-r+1}h_{n-r+1}\cdots g_ng_1'h_2'\cdots h'_{r-1}g'_rh'_r a'_{2r})\\
        &= L(a_{2(n-r)-1}^{-1}h_{n-r}a_{2(n-r)}a_{2(n-r)}^{-1}g_{n-r+1}\cdots g_ng_1'h_2'\cdots g'_ra'_{2r-1}a_{2r-1}^{'-1}h'_r a'_{2r}) \\
        &\le  L(a_{2(n-r)-1}^{-1}h_{n-r}a_{2(n-r)}) \\&+  L(a_{2(n-r)}^{-1}g_{n-r+1}h_{n-r+1}\cdots g_ng_1'h_2'\cdots h'_{r-1}g'_ra'_{2r-1}) \\&+ L(a_{2r-1}^{'-1}h'_r a'_{2r})
    \end{align*}
    by the subadditivity of $L$ on $H$. This works precisely because the middle term is in $A$. Continuing using the subadditivity of $L$ on $G$ and $H$ alternatively, we get that 
    \begin{align*}
        &L(a_{2(n-r)-1}^{-1}h_{n-r}g_{n-r+1}h_{n-r+1}\cdots g_ng_1'h_2'\cdots h'_{r-1}g'_rh'_r a'_{2r})\\
        &\le  L(a_{2(n-r)-1}^{-1}h_{n-r}a_{2(n-r)}) + L(a_{2(n-r)}^{-1}g_{n-r+1}a_{2(n-r)+1}) + \ldots \\&+ L(a_{2n-1}^{-1}g_n) + L(g_1'a_{1}') +\ldots \\&+ L(a_{2r-2}^{'-1}g'_ra'_{2r-1}) + L(a_{2r-1}^{'-1}h'_r a'_{2r}).
    \end{align*}
    Now, combining with our original estimate for $L(kk')$ and applying \ref{afp-len-iden}, we get that $L(kk') \leq L(k) + L(k').$
    \end{enumerate}
    
\end{proof}

Define also a (non-proper) length function on $G*_AH$ as follows: $K(a) = 0$ for all $a\in A$ and $K(g_1h_1\cdots g_nh_n)=2n$ for reduced words. To show that $G*_AH$ has RD/SD, it suffices to show the following. Define the sets $\Lambda_{k,k'}\subset \Gamma$ as follows:
\begin{align*}
    \Lambda_{k,k'} &= \{g\in G*_AH \mid K(g) = k, \ L(g) = k'\}. \\
    \Lambda_{k} &= \{g\in G*_AH \mid K(g) = k\}.
\end{align*}
Note that $K(g)\leq L(g)$ for all $g\in G*_AH$ so that whenever $k > k'$, $\Lambda_{k,k'} = \emptyset$.

\begin{lem}
    Suppose that for all $k,k',\ell,\ell',m$, for all $\varphi\in\bC[G*_AH]$ supported on $\Lambda_{k,k'}$, and for all $\psi\in\bC[G*_AH]$ supported on $\Lambda_{\ell,\ell'}$ we have that 
    $$\|(\varphi*\psi)\chi_{\Lambda_{m}}\|_2 \le f(k')\|\varphi\|_2\|\psi\|_2 $$
    for some polynomial (subexponential) function $f.$ Then $G*_AH$ has RD (resp. SD). 
\end{lem}

\begin{proof}
    The proof is entirely analogous to the part of the proof of Lemma \ref{free_prod_SD} where it is shown that $\bC[G*H]$ satisfies the assumptions of Proposition \ref{prop:jolissaint_1.2.6} since $C_{k'} = \bigcup_{k=1}^{k'} \Lambda_{k,k'}.$ (Here, as in Lemma \ref{free_prod_SD}, $C_{k'}$ denotes the elements $g\in\Gamma$ with $L(g) = k'$.)
\end{proof}

\begin{lem} Suppose $G,H$ have RD (respectively SD) with respect to length functions $L_G,L_H$ respectively. Assume that the restrictions of the length functions of $G$ and $H$ to $A$ agree. Then for all $k,k',\ell,m$, for all $\varphi\in\bC[G*_AH]$ supported on $\bigcup_{j'\le k'}\Lambda_{k,j'}$, and for all $\psi\in\bC[G*_AH]$ supported on $\Lambda_{\ell}$ we have that 
    $$\|(\varphi*\psi)\chi_{\Lambda_{m}}\|_2 \le \tilde{f}(k')\|\varphi\|_2\|\psi\|_2 $$
    for some polynomial (respectively subexponential) function $\tilde{f}$. Moreover, we can take $\tilde{f}(x) = \max(f_G(2x),f_H(2x))$ where $f_G$ and $f_H$ are the decay functions for $G$ and $H$.
\end{lem}

\begin{proof}
    The proof is divided into four cases depending on how much cancellation is required between a $K$-length $k$ and $K$-length $\ell$ element to get a $K$-length $m$ element. This is in analogy to the proofs in \cite{haagerup1978example} and \cite{jolissaint1990rapidly}. Let $f_G$ denote a polynomial (respectively subexponential) function such that $\|x\| \leq f_G(L(x))\|x\|_2$ for any $x\in \bC[G]$, where $L(x)$ is the max of $L(g)$ for $g$ in the support of $x$. We define $f_H$ similarly and set $f = \max(f_G,f_H)$. Note that $A$ also has RD (respectively SD) with respect to $L$ and $f$.

    Case I. Suppose $m = k+\ell$.
    Fix left (respectively right) coset representatives $T_1$ (respectively $T_2$) of $\Lambda_k$ (respectively $\Lambda_\ell$); i.e., $T_1A = \Lambda_k$ and $AT_2 = \Lambda_\ell$. Without loss of generality, pick $T_1$ so that for each $g_1\in T_1,$ $L(g_1) = \min_{a\in A}L(g_1a).$ Then for $s\in \Lambda_m$ there exist unique $g_1\in T_1$, $g_2\in T_2,$ and $b\in A$ such that $s = g_1bg_2$. Moreover, for all $a\in A,$ if $k' \ge L(g_1a)$ then $k' \ge L(g_1)$ and thus $L(a) \le 2k'.$ Hence, we have that
    $$(\varphi*\psi)(s) = \sum_{a\in A, L(a)\le 2k'}\varphi(g_1a)\psi(a^{-1}bg_2)$$
    Therefore, defining $\varphi^{g_1}(a) = \varphi(g_1a)$ for $a\in A$ such that $L(a)\le 2k'$ and 0 otherwise, and defining $\psi_{g_2}(a) = \psi(ag_2)$ for $a\in A$, we have 
    \begin{align*}
        \|(\varphi*\psi)\chi_{\Lambda_m}\|_2^2 &= \sum_{K(s)=m}|(\varphi*\psi)(s)|^2 \\
        &= \sum_{g_i\in T_i, b\in A} \left|\sum_{a\in A, L(a)\le 2k'}\varphi(g_1a)\psi(a^{-1}bg_2)\right|^2\\
        &= \sum_{g_i\in T_i, b\in A} \left|\sum_{a\in A, L(a)\le 2k'} \varphi^{g_1}(a)\psi_{g_2}(a^{-1}b)\right|^2\\
        &= \sum_{g_i\in T_i, b\in A} \left| \varphi^{g_1}*\psi_{g_2}(b)\right|^2\\
        &= \sum_{g_i\in T_i}\|\varphi^{g_1}*\psi_{g_2}\|_2^2\\
\end{align*}

Observe now that $\|\varphi\|_2^2 = \sum_{g_1\in T_1}\|\varphi^{g_1}\|_2^2$ since $T_1$ is a set of right coset representatives for $A$ (and so the supports of each $\varphi^{g_1}$ are disjoint). A similar identity holds for $\psi$ and $T_2$. Now, applying RD (respectively SD) for $A$ and the fact that $\varphi^{g_1}$ is supported on $L(a)\leq 2k'$ we get that

\begin{align*}
        \|(\varphi*\psi)\chi_{\Lambda_m}\|_2^2 &\le \sum_{g_i\in T_i}f(2k')^2\|\varphi^{g_1}\|_2^2\|\psi_{g_2}\|_2^2\\
        &= f(2k')^2\|\varphi\|_2^2\|\psi\|_2^2
    \end{align*}
    and so we conclude that
    $\|(\varphi*\psi)\chi_{\Lambda_{m}}\|_2 \le f(2k')\|\varphi\|_2\|\psi\|_2.$ 

    Case II. Suppose $m=k+\ell -2p$ for some $p\ge1$. We proceed similarly to the proof of Lemma 1.3 in \cite{haagerup1978example} except we track along coset representatives of $A$. As in Case I, choose sets $T_1\subset \Lambda_{k-p}$ and $T_2\subset \Lambda_{\ell-p}$ such that $g_1\in T_1$ implies $L(g_1) = \min_{a\in A}L(g_1a)$ and each $s\in \Lambda_m$ can be written uniquely as $g_1bg_2$ for some $g_1\in T_1,$ $g_2\in T_2$, and $b\in A.$
    
    Let $T\subset\Lambda_p$ be a collection of left coset representatives of $A$ such that $AT=\Lambda_p$ and for all $w\in T$, $L(w) = \min_{a\in A}L(aw)$. For all $v\in\Lambda_p$, there is a unique $a\in A$ and $w\in T$ such that $v = aw$. Note that whenever $L(g_1aw)\le k'$ then $L(a) \le 2k'$ since there exists $c\in A$ such that $L(g_1aw) = L(g_1ac) + L(c^{-1}w)$ and so 
    \begin{align*}
    L(a) &\le L(g_1aw) + L(g_1)+L(w) \\
    &\le k' + L(g_1ac) + L(c^{-1}w) \\
    &= k' + L(g_1aw) \\
    &\le 2k'.
    \end{align*}
    For $t\in \Lambda_{k-p}$, define $$\varphi'(t):= \left(\sum_{ w\in T}|\varphi(tw)|^2 \right)^{1/2}$$
    and $\varphi'(t) = 0$ otherwise;
    similarly, for $u\in \Lambda_{\ell-p}$ define 
    $$\psi'(u):= \left(\sum_{ w\in T}|\psi(w^{-1}u)|^2 \right)^{1/2}$$
    and $\psi'(u)=0$ otherwise.

    Now fix $s\in\Gamma$ and assume that $K(s) = m$ and write $s = g_1bg_2$ as above. Then
    \begin{align*}
        |\varphi*\psi(s)| &= \left|\sum_{K(v)=p}\varphi(g_1v)\psi(v^{-1}bg_2)\right|. \\
    \end{align*}
    But note that terms in the sum are only nonzero when $L(g_1v) \leq k'$, implying that $L(v) \leq L(g_1) + L(g_1v) \leq L(g_1c) + k' \leq 2k'$, where $c\in A$ is such that $L(g_1v) = L(g_1c) + L(c^{-1}v)$. Thus
    \begin{align*}
        |\varphi*\psi(s)| &\le \sum_{K(v)=p, L(v)\le 2k'}|\varphi(g_1v)\psi(v^{-1}bg_2)| \\
        &= \sum_{a\in A, L(a)\le 2k'}\sum_{w\in T}|\varphi(g_1aw)\psi(w^{-1}a^{-1}bg_2)|\\
        &\le \sum_{a\in A, L(a)\le 2k'}\left(\sum_{ w\in T}|\varphi(g_1aw)|^2\right)^{1/2}\left(\sum_{w\in T}|\psi(w^{-1}a^{-1}bg_2)|^2\right)^{1/2}\\
        &= \sum_{a\in A, L(a)\le 2k'} \varphi'(g_1a)\psi'(a^{-1}bg_2) 
    \end{align*}
    As in Case I, we define $\varphi'^{(g_1)}(a) = \varphi'(g_1a)$ for $a\in A$ such that $L(a)\le 2k'$ and 0 otherwise and we define $\psi'_{g_2}(a) = \psi'(ag_2).$ Thus $|\varphi*\psi(s)| \leq \varphi'^{(g_1)} * \psi'_{g_2}(b)$. Applying RD (respectively SD) for $A$ we now have that
    \begin{align*}
        \|(\varphi*\psi)\chi_{\Lambda_m}\|_2^2 &\le \sum_{g_1 \in T_1, g_2\in T_2} \sum_{b\in A} |\varphi'^{(g_1)} * \psi'_{g_2}(b)|^2 \\
        &= \sum_{g_1 \in T_1, g_2\in T_2} \|\varphi'^{(g_1)} * \psi'_{g_2}\|_2^2 \\
        &\le \sum_{g_1 \in T_1, g_2\in T_2} f(2k')^2 \|\varphi'^{(g_1)}\|_2^2 \| \psi'_{g_2}\|_2^2 
    \end{align*}

By a similar argument as in Case I, the supports of $\varphi'^{(g_1)}$ are disjoint for different values of $g_1\in T_1$ (and similarly for $\psi'$). Furthermore, $\|\varphi\|_2 = \|\varphi'\|_2$ since $g_1w = g_1'w'$ for $g_1,g_1'\in \Lambda_{k-p}$ and $w,w'\in T$ is only possible when $w=w'$ and $g_1=g_1'$. We conclude that 
\begin{align*}
    \|(\varphi*\psi)\chi_{\Lambda_m}\|_2 &\le f(2k') \|\varphi'\|_2 \| \psi'\|_2\\
        &= f(2k') \|\varphi\|_2 \| \psi\|_2
\end{align*}
as required.

    Case III. Suppose $m = k+\ell - 1$. As in Case I, fix left (respectively right) coset representatives $T_1$ (respectively $T_2$) of $\Lambda_{k-1}$ (respectively $\Lambda_{\ell-1}$); i.e., $T_1A = \Lambda_{k-1}$ and $AT_2 = \Lambda_{\ell-1}$. Without loss of generality, pick $T_1$ so that for each $g_1\in T_1,$ $L(g_1) = \min_{a\in A}L(g_1a).$ If $s\in\Gamma$ and $K(s)=m,$ we find $g_i\in T_i$ and $t\in G\cup H$ such that $K(g_1) = k-1$, $K(g_2)=\ell-1,$ and $g_1tg_2 = s$. Furthermore, if $xy = s$ where $K(x) = k$ and $K(y) = \ell$ then it must be that $x = g_1v_1$ and $y = v_2g_2$ for some $v_i\in G\cup H$ such that $v_1v_2 = t$. We observe, as in Case II, that if $L(g_1v_1) \le k'$ then $L(v_1)\le 2k'$. We also observe that the last letter of $g_1$ is in $G$ iff the first letter of $g_2$ is in $G$ iff $t\in H$. Therefore we have
    \begin{align*}
        \varphi*\psi(s) &= \sum_{v_i \in G\cup H, v_1v_2 = t, L(v_1)\leq 2k'} \varphi(g_1v_1)\psi(v_2g_2).
    \end{align*}

    Define, for each $g\in T_1,$ a function $\varphi^g$ given by $\varphi^{g}(v) = \varphi(gv)$ for $L(v)\le 2k'$, $v\in G\cup H$ and 0 otherwise. Note that the support of $\varphi^g$ is in $G$ if $K(g) = k-1$ and the last letter of $g$ is in $H$ and vice versa. Similarly, for $g\in T_2$ define $\psi_g(v) = \psi(vg)$ for $v\in G\cup H$. Again, this is supported on $G$ if $K(g) = \ell-1$ and starts with a letter in $H$. Therefore
    \begin{align*}
        \|(\varphi*\psi)\chi_{\Lambda_m}\|_2^2 &= \sum_{K(s)=m}|\varphi*\psi(s)|^2 \\
        &= \sum_{g_1\in T_1, g_2\in T_2}\sum_{t\in G\cup H}\left|\sum_{v_i \in G\cup H, v_1v_2 = t, L(v_1)\leq 2k'}\varphi(g_1v_1)\psi(v_2g_2)\right|^2\\
        &= \sum_{g_1\in T_1, g_2\in T_2}\|\varphi^{g_1}*\psi_{g_2}\|_2^2 \\
    \end{align*}
    We can now use the hypothesis about RD (respectively SD) on $G$ and $H$. We deduce that $\|\varphi^{g_1}*\psi_{g_2}\|_2 \leq f(2k') \|\varphi^{g_1}\|_2\|\psi_{g_2}\|_2$ since $\varphi^{g_1}$ is supported on words with $L(v)\leq 2k'$.  Hence
    \begin{align*}
        \|(\varphi*\psi)\chi_{\Lambda_m}\|_2^2 &\le \sum_{g_1\in T_1, g_2\in T_2}f(2k')^2 \|\varphi_{g_1}\|_2^2\|\psi_{g_2}\|_2^2 \\
        &= f(2k')^2 \|\varphi\|_2^2\|\psi\|_2^2.
    \end{align*}

    Case IV. Now suppose $m = k+\ell - 2p - 1$. Choose as in the previous cases coset representatives $T_1\subset \Lambda_{k-p-1}$ and $T_2\subset \Lambda_{\ell-p-1}$ such that $g_1\in T_1$ implies $L(g_1) = \min_{a\in A}L(g_1a)$. As in Case II, pick a collection of right coset representatives $T\subset \Lambda_p$ such that for all $w\in T$, $L(w) = \min_{a\in A}L(aw)$. For all $v\in\Lambda_p$, there is a unique $a\in A$ and $w\in T$ such that $v = aw$. We note the following for $g_1\in T_1$, $w\in T$, and $v_1\in G\cup H$ such that $v_1$ does not belong to the same group as the last letter of $g_1$ or the first letter of $w$: if $L(g_1v_1w)\le k'$ then $L(v_1)\leq 2k'$. Indeed, there are $a,b\in A$ such that $L(g_1v_1w) = L(g_1a)+L(a^{-1}v_1b)+L(b^{-1}w).$ Therefore 
    \begin{align*}
        L(g_1)+ L(w) &\leq L(g_1a) +L(a^{-1}v_1b)+ L(b^{-1}w) \\
        &= L(g_1v_1w)\\
        &\le k'.
    \end{align*}
    Thus $L(v_1) \le L(g_1v_1w) + L(g_1) +L(w) \le 2k'$. Observe that the conditions on $g_1,$ $w$, and $v_1$ starting and ending with compatible letters as described above is implied by $K(g_1v_1w) = k.$
    
    Define $$\varphi'^{(g_1)}(v):= \left(\sum_{ w\in T}|\varphi(g_1vw)|^2 \right)^{1/2}$$
    for $g_1\in T_1$ and $v\in G\cup H$ with $L(v)\le 2k'$ and $\varphi'^{(g_1)}(v) = 0$ otherwise;
    similarly, define 
    $$\psi'_{g_2}(v):= \left(\sum_{ w\in T}|\psi(w^{-1}vg_2)|^2 \right)^{1/2}$$
    for $g_2\in T_2$ and $v\in G\cup H$ and $\psi'_{g_2}(v) = 0$ otherwise. Note that $\varphi'^{(g_1)}$ is supported on either $G$ or $H$ and also only on $L(v) \le 2k'.$

     If $s\in\Gamma$ and $K(s)=m,$ we find $g_i\in T_i$ and $t\in G\cup H$ such that $K(g_1) = k-p-1$, $K(g_2)=\ell-p-1,$ and $g_1tg_2 = s$. Furthermore, if $xy = s$ where $K(x) = k-p$ and $K(y) = \ell-p$ then it must be that $x = g_1v_1$ and $y = v_2g_2$ for some $v_i\in G\cup H$ such that $v_1v_2 = t$. Now fix some $s\in\Gamma$ and assume that $K(s) = m$. 

    \begin{align*}
        |\varphi*\psi(s)| &= \left|\sum_{v_i\in G\cup H, v_1v_2=t, L(v_1)\le 2k'}\sum_{w\in T}\varphi(g_1v_1w)\psi(w^{-1}v_2g_2)\right| \\
        &\le \sum_{\substack{v_i\in G\cup H, v_1v_2=t, \\ L(v_1)\le 2k'}} \left(\sum_{w\in T}|\varphi(g_1v_1w)|^2\right)^{1/2}\left(\sum_{w\in T}|\psi(w^{-1}v_2g_2)|^2\right)^{1/2}\\
        &= \varphi'^{(g_1)} * \psi'_{g_2}(t).
    \end{align*}
    
We therefore have that

\begin{align*}
    \|(\varphi*\psi)\chi_{\Lambda_m}\|_2^2 &= \sum_{g_i\in T_i} \sum_{t \in G\cup H} |\varphi*\psi(g_1tg_2)|^2 \\
    &\le \sum_{g_i\in T_i} \sum_{t \in G\cup H}\varphi'^{(g_1)} * \psi'^{(g_1)}_{g_2}(t)\\
    &\le \sum_{g_i\in T_i} \|\varphi'^{(g_1)} * \psi'_{g_2}\|_2^2.
\end{align*}
We recall that each $\varphi'^{(g_1)}$ and $\psi'_{g_2}$ is defined only on $G$ or on $H$; furthermore, $\varphi'^{(g_1)}$ is supported only on $v_1$ such that $L(v_1)\le 2k'$, so we can use RD (respectively SD) of $G$ (and $H$) to deduce that

\begin{align*}
    \|(\varphi*\psi)\chi_{\Lambda_m}\|_2^2 &\le  \sum_{g_i\in T_i} f(2k')^2\|\varphi'^{(g_1)}\|_2^2 \|\psi'_{g_2}\|_2^2.
\end{align*}

Now observe that if $g_2,g_2'\in T_2$, $w,w'\in T$, and $v,v'\in G\cup H$ are such that $K(w^{-1}vg_2) = K(w'^{-1}v'g_2') = k$ then since we have fixed cosets for $w,w'$ and $g_2,g_2'$ it must be that $w=w',$ $g_2=g_2'$, and $v=v'$. This implies that
$$\sum_{g_2\in T_2}\|\psi'_{g_2}\|_2^2 = \|\psi\|_2^2. $$
We get a similar identity for $\varphi$ and so we can conclude that
\begin{align*}
    \|(\varphi*\psi)\chi_{\Lambda_m}\|_2 &\le f(2k')\|\varphi\|_2\|\psi\|_2
\end{align*}

    In all cases, we achieve a polynomial (respectively subexponential) function $\tilde{f}$, namely $f(2k')$, in $k'$ as required.
\end{proof}

The above lemmas prove the following theorem.

\begin{thm}[Theorem \ref{thm: no distortion}]
    Suppose $(G,L_G)$ and $(H,L_H)$ have RD (respectively SD). Suppose $A$ is a common subgroup of $G$ and $H$ such that $L_G|_A = L_H|_A$. Set $\Gamma=G*_AH$. Then $(\Gamma ,L)$ has RD (respectively SD), where $L$ is the length function in Definition \ref{def: afp-len-agree-on-A}.
\end{thm}

\begin{rem}
\label{rem: distortion}
    Even if $L_G$ and $L_H$ do not agree on $A$, one can still define a ``universal'' length function $L^U$ on the amalgamated free product $G*_A H$ as follows:
    $$L^U(k) = \min\left\{\sum_i L_i(k_i) : \prod_i k_i = k, k_i\in G\cup H\right\},$$
    where $L_i = L_G$ if $k_i\in G\setminus A$, $L_i= L_H$ if $k_i\in H\setminus A$, and if $k_i$ is in $A$, then $L_i$ can be chosen either to be $L_G$ or $L_H$ (this is key). Now, since $L^U$ is defined on all of $G*_AH$, its restriction to $G$ and $H$ now gives length functions which agree on $A$. So if $G,H$ have RD with respect to the restrictions of $L^U$ then the above theorem applies. This will happen if $L^U$ is not too \emph{distorted} relative to $L_G$ and $L_H$. Say two length functions $(L_1,L_2)$ on a group $K$ are $f$-distorted if $L_1(k)\le f(L_2(k)) $ and $L_2(k)\le f(L_1(k))$ for all $k\in K$. 
    
    It is straightforward from the definitions and from composition of functions that the following hold:
    \begin{itemize}
        \item If $(L_1,L_2)$ are polynomially distorted then $(K,L_1)$ has RD iff $(K,L_2)$ has RD;
        \item If $(L_1,L_2)$ are subexponentially distorted and $(K,L_1)$ has RD, then $(K,L_2)$ has SD;
        \item If $(L_1,L_2)$ are linearly distorted (so they are coarsely equivalent) then $(K,L_1)$ has SD iff $(K,L_2)$ has SD.
    \end{itemize}
\end{rem}

Theorem \ref{thm: no distortion} therefore generalizes as follows.

\begin{thm}[Theorem \ref{thm:afp-distortion}]
    Let $(G,L_G)$ and $(H,L_H)$ be groups with length functions and a common subgroup $A$. Define $L^U$ as in Remark \ref{rem: distortion}. Then $G*_AH$ has
    \begin{enumerate}
        \item RD if $(G,L_G)$ and $(H,L_H)$ have RD and the pairs $(L^U|_G, L_G)$ on $G$ and $(L^U|_H,L_H)$ on $H$ are each polynomially distorted;
        \item SD if $(G,L_G)$ and $(H,L_H)$ have RD and the pairs $(L^U|_G, L_G)$ on $G$ and $(L^U|_H,L_H)$ on $H$ are each subexponentially distorted;
        \item SD if $(G,L_G)$ and $(H,L_H)$ have SD and the pairs $(L^U|_G, L_G)$ on $G$ and $(L^U|_H,L_H)$ on $H$ are each linearly distorted.
    \end{enumerate}
\end{thm}

Now we prove the last two stated results in the introduction. 

\begin{prop}[Proposition \ref{amalgam-distort-bounds-afp-distort}]
    Let $(G,L_G)$ and $(H,L_H)$ have common subgroup $A$. Suppose that $(L_G|_A,L_H|_A)$ are $f$-distorted on $A$. Then $(L^U|_G,L_G)$ and $(L^U|_H,L_H)$ are $\tilde{f}$-distorted where
    \begin{enumerate}
        \item if $f(x) = x + O(1)$, then $\tilde{f}$ can be taken linear;
        \item if $f(x) = x+ O(\ln(x))$, then $\tilde{f}$ can be taken polynomial;
        \item if $f(x) = x + o(x)$, then $\tilde{f}$ can be taken subexponential.
    \end{enumerate}
\end{prop}

\begin{proof}
    Suppose $(L_G|_A,L_H|_A)$ are $f$-distorted on $A$. It suffices, by symmetry, to prove that $(L^U|_G,L_G)$ are $\tilde{f}$-distorted on $G$. It is clear that $L^U|_G \le L_G$ so we simply need to show that $L_G(g) \le \tilde{f}(L^U(g))$.

    Write $f(x) = x + j(x)$ for some $j(x) \in o(x)$. Without loss of generality, assume that $j(x)\ge0$, $j(x)$ is increasing, and that $j(x)/x$ is decreasing. (If $j(x) \in O(1)$ then simply replace $j(x)$ with $C$ a constant. If $j(x) \in O(\ln(x))$ then replace $j(x)$ with $C\ln(x)$ for an appropriate $C$. Otherwise, if $j(x) \in o(x)$ then there are $x_n$ such that $j(x) \leq \frac{x}{n}$ for all $x\geq x_n$. Define ${j}'(x) = \frac{x}{n}$ for $x\in[x_{n},x_{n+1})$. Then define $j''(x) = \max_{y\le x} j'(x)$. $j''(x)$ now dominates $j(x)$, but is still in $o(x)$, increasing, and $j''(x)/x$ decreases to 0.)
    
    Note that the conditions on $j(x)$ implies that $x\le f(x)$ for all $x$ and that $x+f(y) \leq f(x+y)$ for all $x,y.$ Denote by $f^{(k)}$ the $k$-fold composition of $f$ with itself. It then follows that for $h,h'\in H$ and $a\in A$ that if $hah'\in A$, then
    \begin{align*}
        L_G(hah') &\le f(L_H(hah')) \\
        &\le f(L_H(h) + L_H(h') + L_H(a))\\
        &\le f(L_H(h) + L_H(h') + f(L_G(a)))\\
        &\le f^{(2)}(L_H(h) + L_H(h') + L_G(a)).
    \end{align*}
    Now fix $g\in G$ and write $g = \prod_{i=1}^n k_i$ for $k_i\in G\cup H$. Pick $L_i$ such that $L_i = L_G$ if $k_i\in G\setminus A,$ $L_i = L_H$ if $k_i\in H\setminus A$, and if $k_i\in A$ pick $L_i$ to be either $L_G$ or $L_H$. Pick the $k_i$ and $L_i$ such that the sum $\sum_i L_i(k_i)$ is minimized; i.e., so that $L^U(g) = \sum_{i=1}^{n} L_i(k_i)$. Note that since the $L_i$ are integer-valued, we can assume $n \leq L^U(g)$. Furthermore, we can always assume that two consecutive $k_i$ never both come from $H\setminus A$, since if $h,h'\in H\setminus A$ then $L_H(hh') \leq L_H(h) + L_H(h')$. Furthermore, since $\prod_i k_i \in G$ it must be that all terms coming from $H\setminus A$ cancel, in the sense that if $i_1<i_2$ and $k_{i_1},k_{i_2}\in H\setminus A$, then $\prod_{i=i_1}^{i_2}k_i \in A$. Let $i_1,\ldots,i_m$ be the indices such that $k_{i_j}\in H\setminus A$. Then
    \begin{align*}
        L_G(g) &= L_G\left(\prod_{i=1}^n k_i\right) \\
        &\le L_G(k_1) + \ldots + L_G(k_{i_1-1}) + L_G\left(\prod_{i= i_1}^{i_m}k_i\right) + L_G(k_{i_m+1}) + \ldots + L_G(k_n) \\
        &\le L_G(k_1) +\ldots + L_G(k_{i_1-1})\\
        &+ f^{(2)}\left(L_H(k_{i_1}) + L_G\left(\prod_{i=i_1+1}^{i_m-1}k_i\right)+ L_H(k_{i_m})\right) \\
        &+ L_G(k_{i_m+1}) + \ldots+  L_G(k_n)\\
        &\le f^{(2)}\left(\ldots + L_H(k_{i_1}) + L_G\left(\prod_{i=i_1+1}^{i_m-1}k_i\right) + L_H(k_{i_m}) + \ldots\right)
\end{align*}
Continuing recursively on each pairing of the $k_i\in H\setminus A$, we deduce that
\begin{align*}
        L_G(g) &\le f^{(m)}\left(\sum_{k_i\in G} L_G(k_i) + \sum_{k_i \in H\setminus A} L_H(k_i)\right) \\
        &\le f^{(m)}\left(\sum_{k_i\in G\setminus A} L_G(k_i) + \sum_{k_i\in A} f(L_i(k_i))+ \sum_{k_i \in H\setminus A} L_H(k_i)\right)\\
        &\le f^{(n)}\left(\sum_{i=1}^n L_i(k_i)\right) \\
        &\le f^{(L^U(g))}(L^U(g)).
    \end{align*}
    In the case that $j(x) = C$, $f^{(L^U(g))}(L^U(g)) = L^U(g)(C+1)$ and so we have linear distortion. Otherwise, note that we have the following estimates.

    \begin{align*}
        f^{(N)}(N) &= f(f^{(N-1)}(N)) \\
        &= f^{(N-1)}(N) + j(f^{(N-1)}(N)) \\
        &= f^{(N-1)}(N) \left(1 + \frac{j(f^{(N-1)}(N))}{f^{(N-1)}(N)}\right) \\
        &\le f^{(N-1)}(N)\left(1+\frac{j(N)}{N}\right)
    \end{align*}
    where we used in the inequality the fact that $f(x) \geq x$ for all $x$ and that $j(x)/x$ is decreasing. This implies that 
    $$f^{(N)}(N) \le N\left(1+\frac{j(N)}{N}\right)^N \le N\exp(j(N)).$$
    
    Thus $L_G(g) \le L^U(g)\exp(j(L^U(g)))$; that is, we have polynomial distortion (of degree $C+1$) if $j(x) = C\ln(x)$ and subexponential distortion if $j(x) \in o(x)$.
    
\end{proof}

Proposition \ref{amalgam-distort-bounds-afp-distort} and Theorem \ref{thm:afp-distortion} immediately imply the following.

\begin{cor}[Corollary \ref{thm: amalgam-distortion}]
    Let $(G,L_G)$ and $(H,L_H)$ have common subgroup $A$. Suppose that $(L_G|_A,L_H|_A)$ are $f$-distorted on $A$.
    \begin{enumerate}
        \item If $f(x) = x + O(\ln(x))$ and $G,H$ have RD then $G*_AH$ is RD;
        \item If $f(x) = x + o(x)$ and $G,H$ have RD then $G*_AH$ is SD;
        \item If $f(x) = x + O(1)$ and $G,H$ have SD then $G*_AH$ has SD.
    \end{enumerate}
\end{cor}

\bibliographystyle{amsplain}
\bibliography{comparison.bib}

\end{document}